\documentclass[11pt,notitlepage,twoside,a4paper]{amsart}
 \usepackage{amsfonts}
 \usepackage{mathrsfs}

\usepackage{amsmath,amssymb,enumerate}

\usepackage{epsfig,fancyhdr,color}

\usepackage{amssymb}
\usepackage{amsmath,amsthm}
\usepackage{latexsym}
\usepackage{amscd}
\usepackage{psfrag}
\usepackage{graphicx}
\usepackage[latin1]{inputenc}
\usepackage[all]{xy}

\usepackage{tikz}
\usepackage{pgflibraryarrows}
\usepackage{pgflibrarysnakes}

\newtheorem{theorem}{\rm\bf Theorem}[section]
\newtheorem{proposition}[theorem]{\rm\bf Proposition}
\newtheorem{lemma}[theorem]{\rm\bf Lemma}
\newtheorem{corollary}[theorem]{\rm\bf Corollary}

\theoremstyle{definition}
\newtheorem{definition}[theorem]{\rm\bf Definition}

\theoremstyle{remark}
\newtheorem{remark}[theorem]{\rm\bf Remark}

\def\interieur#1{\mathord{\mathop{\kern 0pt #1}\limits^\circ}}

\title[Variation of extremal length]{Variation of extremal length functions on Teichm\"uller space}

\author{Lixin Liu}
\address{Lixin Liu, Department of Mathematics, Sun Yat-sen University, 510275, Guangzhou, P. R. China}
\email{mcsllx@mail.sysu.edu.cn}

\author{Weixu Su}
\address{Weixu Su, School of Mathematics and Shanghai Center of Mathematics, Fudan University, 200433, Shanghai, P. R. China}
\email{suwx@fudan.edu.cn}

\date{\today}

\begin{document}

\begin{abstract}
Extremal length is an important conformal invariant on Riemann surface.
It is  closely related to the geometry of Teichm\"uller metric on Teichm\"uller space.
By identifying extremal length functions with energy of harmonic maps from Riemann surfaces to $\mathbb{R}$-trees,
we study the second variation of extremal length functions along Weil-Petersson geodesics.
We show that the extremal length of any measured foliation is a
pluri-subharmonic function on Teichm\"uller space.
\end{abstract}

\maketitle


\medskip

\medskip

\maketitle

\tableofcontents
\section{Introduction}
Let $S$ be a smooth closed surface of negative Euler characteristic.
The Teichm\"uller space $\mathcal{T}(S)$ is the space of isotopy classes of marked hyperbolic structures on $S$.
The Teichm\"uller space admits a canonical complex structure and a standard distance,
called the Teichm\"uller ($=$ Kobayashi) distance.
Thanks to the works of Kerckhoff, Masur and Gardiner, etc.,
the extremal length functions, which are studied in this paper,
are recently recognized as fundamental tools for studying the geometry of the Teichm\"uller distance.

The extremal length functions are often comparable with the hyperbolic length functions, which are also fundamental in Teichm\"uller theory.
As was observed by Wolpert and other mathematicians, the hyperbolic length functions are plurisubharmonic.
Hence it was also conjectured and expected that the extremal length functions are  plurisubharmonic.
In this paper, we give an affirmative answer to this conjecture.


\subsection{Motivation and main results}
Let $\gamma$ be a fixed simple closed curve on $S$.
Denote by
$$\mathrm{Ext}_\gamma(\cdot): \mathcal{T}(S)\to \mathbb{R}_+$$
the extremal length function of $\gamma$.
The notion of extremal length is due to Ahlfors and Beurling \cite{AB}.
We will give various equivalent definitions of extremal length on \S \ref{sec:pre}.
Kerckhoff \cite{Kerckhoff} discovered a useful distance formula for the Teichm\"uller metric in terms of extremal length functions.
Estimates of extremal length along Teichm\"uller geodesics are important to understand the large scale geometry  of the Teichm\"uller metric \cite{Masur75, Minsky,Rafi, FM}.

In this paper, we initiate the study of second variation of extremal length functions on $\mathcal{T}(S)$.
Note that the first variational formula of extremal length functions along a differential path on Teichm\"uller space has been known for a long time
(see for example Gardiner \cite{Gardiner}.
A new proof  using harmonic maps was recently given by Wentworth \cite{Wentworth}).

By the work of Wolf  \cite{Wolf95,Wolf96}, for any $X\in \mathcal{T}(S)$,
$\mathrm{Ext}_\gamma(X)$ can be identified with the energy of a harmonic map from $X$ to a $\mathbb{R}$-tree.
To be more precise,
there is a $\pi_1(S)$-equivariant harmonic map from $\widetilde{X}$, the universal cover
 of $X$, to a $\mathbb{R}$-tree determined by the leaf structure of the measured foliation equivalent to $\gamma$.
The Hopf differential of such a harmonic map realizes the Hubbard-Masur differential of $\gamma$,
and the $\mathbb{R}$-tree is  dual to the vertical foliation of the Hopf differential.
 It turns out that the energy of the harmonic map (restricted on a fundamental domain of $X$) is equal to $\mathrm{Ext}_\gamma(X)$.
See \S \ref{sec:pre} for more details.

 The above observation allows us to investigate the variations of $\mathrm{Ext}_\gamma(\cdot)$ via harmonic maps.
Our main result is:

\begin{theorem}\label{theorem:WP} Let $X\in \mathcal{T}(S)$ and assume that $\mathrm{Ext}_\gamma(\cdot)$
is smooth in an open  neighborhood of $X$.
For any two Weil-Petersson geodesics $\Gamma_i(t),i=1,2$ on $\mathcal{T}(S)$ with
$\Gamma_1(0)=\Gamma_2(0)=X$ and $\frac{d}{dt}\Gamma_1(0)={\mu}$, $\frac {d}{dt}\Gamma_2(0)=i{\mu}$,
where $\mu$ is a harmonic Beltrami differential representing a tangent vector to
$\mathcal{T}(S)$ at $X$ and $i$ denotes the almost complex structure on $\mathcal{T}(S)$,
we have
$$\frac{d^2} {dt^2}|_{t=0}\mathrm{Ext}_\gamma(\Gamma_1(t))+\frac{d^2}{dt^2}|_{t=0} \mathrm{Ext}_\gamma(\Gamma_2(t))>0.$$
\end{theorem}
Note that extremal length is a conformal invariant, while the Weil-Petersson metric is defined by using hyperbolic geometry.
Hence Theorem \ref{theorem:WP} gives some analytic characterization of the extremal length functions on Teichm\"uller  space in terms of hyperbolic (Weil-Petersson) geometry.

As a corollary of the the above theorem, we  show that:

\begin{theorem}\label{theorem:main}Let $\gamma$ be a simple closed curve on $S$.
The extremal length function $\mathrm{Ext}_{\gamma}(\cdot)$ is pluri-subharmonic on $\mathcal{T}(S)$.
It is smooth and strictly pluri-subharmonic on an open dense subset of $\mathcal{T}(S)$.
\end{theorem}

Harmonic map theory has been successfully used by Wolf \cite{Wolf89}
in the study of the Thurston compactification and
the Weil-Petersson geometry (see also \cite{SW}, where Scannell and Wolf studied
infinitesimal variation of harmonic maps between singular grafted surfaces).
 Our work shows that harmonic maps may be useful to the study of extremal length functions and Teichm\"uller metric.

In contrast with extremal length, variational formulas for  hyperbolic length were throughly studied by
Kerckhoff \cite{Kerckhoff83}, Wolpert \cite{Wolpert83,Wolpert87,Wolpert08} and Wolf \cite{Wolf09}.
The convexity of hyperbolic length was used to solve the famous Nielsen realization problem \cite{Kerckhoff83,Wolpert87}.
The estimates of  Weil-Petersson gradient and  Hessian  of hyperbolic length are
important in the study of Weil-Petersson curvature expansions and geodesic flow \cite{Wolpert10, BMW}.

\subsection{Remark}
Note that Tromba has proved that the Dirichlet energy function
(with varied domain and fixed target) is strictly pluri-subharmonic on $\mathcal{\mathcal{T}}(S)$. More explicitly, fix a hyperbolic structure $g_0$ on $S$, for any  $X\in \mathcal{T}(S)$, there exists a unique harmonic map $h:X\to (S,g_0)$ homotopic to the identity map of $S$.
Denote the energy of $h$ by $D(X)$. It was shown by Tromba  \cite{Tromba,Tromba2} that $D(\cdot)$ defines a strictly pluri-subharmonic function on $\mathcal{\mathcal{T}}(S)$.

The fact that the target surface $(S,g_0)$ has constant curvature $-1$ plays an important role in Tromba's argument,
since integral of curvature form appears in the estimation of
Levi form.
While in our situation, the target is flat with singularity.
However, we can make some analogous computations by restricting the integral of energy density on a
cylinder domain, whose image under the harmonic maps is isometric to the unit interval.
Such a one-dimensional feature will simplify our computations.

On the other hand, the Dirichlet energy function, with fixed domain and varied target, is convex along
Weil-Petersson geodesics (see \cite{Wolf89, Tromba, Yamada}).
In this case, the Hessian of Dirichlet energy function realizes  the Weil-Petersson Riemannian metric.
It is then possible to develop Teichm\"uller theory in terms of harmonic maps by systematic investigation of variations of the target hyperbolic structures, see Wolf \cite{Wolf89} and Jost \cite{Jost} for references.

\subsection{Organization of the paper}

We shall recall Wolf's treatment of extremal length functions as energy of harmonic maps from Riemann surfaces to  $\mathbb{R}$-trees in \S \ref{sec:pre}.
In \S \ref{sec:second}, we obtain a formal formula for the second variation of extremal length functions  on Teichm\"uller space.
Such a formula is simplified in \S \ref{sec:field}, where we study the variational vector of harmonic maps.
The computations in \S \ref{sec:second} make sense when the extremal length function is smooth. Unfortunately,
in general, extremal lengths functions can be even not $C^2$.  Some of the regularity results are discussed in \S \ref{sec:regularity}.
Then \S \ref{sec:pluri} is devoted to the proof of Theorem \ref{theorem:WP} and Theorem \ref{theorem:main}.
Some applications are given in \S \ref{sec:application}, including a  new proof of the classical result that
Teichm\"uller space is a Stein manifold.

\bigskip

\emph{Note Added in Proof.} The first version of this paper was posted on
arXiv \cite{LS}, where we announced  that the extremal length functions of simple closed curves are  strictly plurisubharmonic.
The proof in that version was incomplete, due to the lack of smoothness (at least $C^2$) of extremal length functions.
Now we recast our result to Theorem \ref{theorem:main}.
Recently Miyachi \cite{Miyachi} showed that extremal length functions are log-plurisubharmonic, which implies our main theorems.
His argument is more directly, by giving an explicit formula for the Levi form of the norm of Hubbard-Masur differentials.

\bigskip
\emph{Acknowledgements.}  We are grateful to
Michael Wolf for his  ideas and encouragements. 
We thank  Howard Masur for communication on the proof of Lemma \ref{lem:denseopen}.
We also would like to thank the referee for careful reading of the manuscript,
with many corrections and useful comments. The first author is partially supported by NSFC No: 11271378. The second author is partially supported by NSFC No: 11671092, 11631010 and the Heidelberg Institute for Theoretical Studies.

\section{Preliminaries}\label{sec:pre}
Let $S$ be a connected smooth closed surface of genus $g>1$.
Let $\mathcal{M}_{-1}$ be the space of hyperbolic metrics
(complete Riemannian metrics of constant curvature $-1$)
on $S$  and let $\mathrm{Diff}_0$ be the group of diffeomorphisms of $S$
isotopic to the identity.
The {Teichm\"uller space} $\mathcal{T}(S)$ of $S$ is defined to be the quotient space
$\mathcal{M}_{-1}/{\mathrm{Diff}_0}$,
where ${\mathrm{Diff}_0}$ acts on $\mathcal{M}_{-1}$ by pulling back.

 Note that every hyperbolic metric on $S$ is corresponding to a conformal structure.
 A surface with a conformal structure is a {Riemann surface}.
 The Teichm\"uller space $\mathcal{T}(S)$
 is also the set of equivalence classes of marked Riemann surfaces homeomorphic to $S$.
 For further background on Teichm\"uller theory we refer to the book \cite{GL}.

Throughout the paper, we shall identify a hyperbolic metric on $S$
with its corresponding conformal structure.
A hyperbolic metric on $S$ is usually denoted by
$(S,g)$ where, in  conformal coordinates, $g$ is locally of the form $g=g(z)|d z|^2$.
Furthermore, we shall denote an element of $\mathcal{T} (S)$ by $(S,g)$,
without explicit reference to the equivalence relation.

\subsection{Extremal length}
Let $X$ be a  Riemann surface. A \emph{conformal metric} $\rho$ on $X$ is  locally of the form
$\rho(z)|dz|$ where $z$ is a locally conformal coordinate of $X$ and $\rho(z)\geq 0$ is
a  Borel measurable function.  We define the
\emph{$\rho$-area} of $X$ by
$$ \mathrm{Area}_\rho(X)=\int_X \rho^2(z)|dz|^2.$$

In the following, we shall denote by $\gamma$ an essential simple closed curve on $S$ or its free isotopy class on $S$
(essential means not homotopic to a point).
The \emph{$\rho$-length} of $\gamma$ is defined by
$$L_\rho(\gamma)=\inf_{\gamma' }\int_{\gamma'}\rho(z)|dz|,$$
where the infimum is taken over all simple closed curves $\gamma' $ in the isotopy class of  $\gamma$.

\begin{definition}
With the above notation, the {extremal length} of $\gamma$ on $X$ is defined by
\begin{equation*}\label{equ:extremal}
\mathrm{Ext}_\gamma(X)=\sup_\rho \frac{L_\rho^2(\gamma)}{\mathrm{Area}_\rho(X)},
\end{equation*}
where  $\rho(z)|dz|$ ranges over all conformal metrics on $X$  with
$$0<\mathrm{Area}_\rho(X)<\infty.$$
\end{definition}

The above definition is called the \emph{analytic definition} of extremal length.
There is a \emph{geometric definition} of $\mathrm{Ext}_\gamma(X)$:
$$\mathrm{Ext}_\gamma(X)=\inf_C \frac{1}{\mathrm{mod}(C)},$$
where the infimum is taken over all embedded cylinders $C$ in $X$ with core curves isotopic to $\gamma$
and $\mathrm{mod}(C)$ is the conformal modulus of $C$. As pointed out by Kerckhoff \cite{Kerckhoff},
in the estimation of extremal length, the analytic definition is useful for finding lower bounds, while the geometric definition is useful for finding upper bounds.

\subsection{Measured foliation and Quadratic differential}
A \emph{measured foliation}  on $S$ is a foliation  (with a finite
number of singularities) endowed with a transversely invariant  measure.
 The allowed
singularities are  topologically  the same as
those that occur at $z=0$ in the line field $z^{p-2}dz^2, p\geq 3$. The \emph{intersection number}
$i(\gamma, \mathcal{F})$
of a simple closed curve $\gamma$ with a measured foliation $\mathcal{F}$ endowed with transverse measure $\mu$
is defined by
$$i(\gamma,\mathcal{F})=\inf_{\gamma'}\int_{\gamma'}d\mu,$$
where the infimum is taken over all simple closed curves $\gamma' $ in the isotopy class of  $\gamma$.
Two measured laminations $\mathcal{F}$ and $\mathcal{F}'$ are said to be measure
equivalent if, for all simple closed curves $\gamma$ on $S$,  $i(\gamma,\mathcal{F})=i(\gamma,\mathcal{F}')$.
Denote by $\mathcal{MF}$
the space of equivalence classes of measured foliations on $S$.

Given a measured foliation, the leaves passing through a singularity are the \emph{critical leaves}.
There is a special class of measured foliations with the property that the complement of the critical leaves is
homeomorphic to a cylinder. For such a measured foliation, the leaves  on the
cylinder are all freely homotopic to some simple closed curves
$\gamma$. Then the measured foliation is completely determined as a point in
$\mathcal{MF}$ by the height $a$ of the cylinder
 and the isotopy class of $\gamma$.
 Denote such a foliation by $(\gamma, a)$ or $a\gamma$.

 Let $\mathcal{S}$ be the set of isotopy classes of essential simple closed curves on $S$.
 Thurston (see \cite{FLP}) showed that $\mathcal{MF}$ is homeomorphic to a open ball of dimension $6g-6$
and  the  embedding $\mathcal{S}\times \mathbb{R}_+\to \mathcal{MF}$  is dense in $\mathcal{MF}$.

\bigskip

A  \emph{holomorphic quadratic
differential} $\Phi$ on $(S,g)$ is a $(2,0)$-tensor locally given by $\Phi=\Phi(z)dz^2$,
where $\Phi(z)$ is holomorphic. Any holomorphic quadratic differential $\Phi=\Phi(z)dz^2$
determines a singular metric $|\Phi(z)||dz|^2$,
with finitely many singular points corresponding to the zeros of $\Phi$. The total area of $S$ in this metric is given by
$$\int_S |\Phi|= \int_S |\Phi(z)|dz|^2.$$
Let $QD(g)$ be the space of
holomorphic quadratic differentials on $(S,g)$.
There is a natural identification of $QD(g)$ with the holomorphic cotangent space of $\mathcal{T}(S)$ at $(S,g)$.

Let us describe the relation between quadratic differentials and measured foliations.
Firstly, an element $\Phi \in QD(g)$ gives rise to a pair of transverse
measured foliations $h(\Phi)$ and $v(\Phi) $ on
$S$, called the \emph{horizontal foliation} and \emph{vertical foliation} of $\Phi$, respectively. The
leaves of these foliations are given by setting the imaginary part
(resp. real part) of $\Phi$ equal to a constant. In a neighborhood of a
nonsingular point, there are natural coordinates $z = x + iy $ so
that the leaves of $h(\Phi)$ are given by $y$ = constant, and
the transverse measure of $h(\Phi)$ is $|dy|$. And the leaves of
$v(\Phi)$ are given by $x$ = constant, with transverse
measure $|dx|$. The foliations $h(\Phi)$ and
$v(\Phi) $ have zero set of $\Phi$ as their common singular
set, and at each zero of order $k$ they have a $k+2$-pronged
singularity, locally modeled on the singularity at the origin of
$z^k dz^2$.

Conversely, according to a fundamental theorem of Hubbard and Masur \cite{HM}, if $\mathcal{F}$ is a measured foliation on $(S,g)$, then
there is a unique holomorphic quadratic differential $\Phi(\mathcal{F})\in QD(g)$ such that
$\mathcal{F}$ is measure equivalent to $v(\Phi(\mathcal{F}))$, the vertical measured foliation of $\Phi(\mathcal{F})$.
$\Phi(\mathcal{F})$ is called the \emph{Hubbard-Masur differential} of $\mathcal{F}$.
A quadratic differential whose vertical foliation is measure equivalent to a $(\gamma, a)$
is called a \emph{one-cylinder Strebel differential}.


For any simple closed curve $\gamma$ on $S$ and $a>0$, we set $$\mathrm{Ext}_{a\gamma}(X)=a^2\mathrm{Ext}_\gamma(X).$$ Based on the result that $\mathcal{S}\times \mathbb{R}_+$ is dense in $\mathcal{MF}$, Kerckhoff  \cite{Kerckhoff} generalized the definition of extremal length of simple closed curves to that of measured foliations.

The following fact is due to Kerckhoff \cite{Kerckhoff}. A detailed proof can be found in Ivanov \cite{Ivanov}.

\begin{proposition}\label{pro:area} The extremal length of any measured foliation $\mathcal{F}$ on $(S,g)$ is equal to the area of
$\Phi(\mathcal{F})$, that is,
$$\mathrm{Ext}_{\mathcal{F}}(g)=\int_S|\Phi(\mathcal{F})|.$$
\end{proposition}
\label{pro1}





\subsection{Realizing extremal length by energy of harmonic map}

\begin{figure}[!hbp]
\centering
\begin{tikzpicture}[scale=0.7]
\path (0,0) coordinate (origin);
\path (0:2cm) coordinate (P0);
\path (1*72:2cm) coordinate (P1);
\path (2*72:2cm) coordinate (P2);
\path (3*72:2cm) coordinate (P3);
\path (4*72:2cm) coordinate (P4);
\path (0:10cm) coordinate (Q);
\path (36:2cm) coordinate (Q0);
\path (108:2cm) coordinate (Q1);
\path (180:2cm) coordinate (Q2);
\path (252:2cm) coordinate (Q3);
\path (-36:2cm) coordinate (Q4);
\path (6:11.5cm) coordinate (R0);
\path (11:9cm) coordinate (R1);
\path (0:8cm) coordinate (R2);
\path (-13:9cm) coordinate (R3);
\path (-6:11.5cm) coordinate (R4);

\path (10:2.05cm) coordinate (V1);
\path (62:2.05cm) coordinate (V2);
\path (82:2.05cm) coordinate (V3);
\path (134:2.05cm) coordinate (V4);
\path (154:2.05cm) coordinate (V5);
\path (206:2.05cm) coordinate (V6);
\path (226:2.05cm) coordinate (V7);
\path (-82:2.05cm) coordinate (V8);
\path (-62:2.05cm) coordinate (V9);
\path (-10:2.05cm) coordinate (V10);
\draw  (origin) -- (P0) (origin) -- (P1) (origin) -- (P2)
(origin) -- (P3) (origin) -- (P4);
\draw [thick][red] (origin) -- (Q0)  (origin) -- (Q1) (origin) -- (Q2) (origin) -- (Q3) (origin) -- (Q4);
\draw [thick][red] (Q) -- (R0)  (Q) -- (R1) (Q) -- (R2) (Q) -- (R3) (Q) -- (R4);
\draw (10:2.05cm) -- (-10:2.05cm);
\draw  (82:2.05cm) -- (62:2.05cm);
\draw  (154:2.05cm) -- (134:2.05cm);
\draw  (226:2.05cm) -- (206:2.05cm);
\draw  (-62:2.05cm) -- (-82:2.05cm);
\draw (V1) .. controls  (10:1cm) and (62:1cm) .. (V2);
\draw  (V3) .. controls  (82:1cm) and (134:1cm) .. (V4);
\draw  (V5) .. controls  (154:1cm) and (206:1cm) .. (V6);
\draw  (V7) .. controls  (226:1cm) and (-82:1cm) .. (V8);
\draw (V9) .. controls  (-62:1cm) and (-10:1cm) .. (V10);
\draw  (5:2.03cm) .. controls  (5:0.5cm) and (67:0.5cm) .. (67:2.03cm);
\draw  (7.5:2.03cm) .. controls  (2.5:0.75cm) and (69.5:0.75cm) .. (64.5:2.03cm);
\draw  (2.5:2.03cm) .. controls  (2.5:0.25cm) and (69.5:0.25cm) .. (69.5:2.03cm);
\draw [thick] (origin) circle (4cm);
\draw [ultra thick][->] (4.2,0) -- (5.5,0);
\draw [thick] (10,0) circle (4cm);
\node [above] at (5,0) {$\pi$};
\end{tikzpicture}
\caption{\small {The local picture of the projection $\pi$ on the universal cover.}}
\label{fig:leaf}
\end{figure}
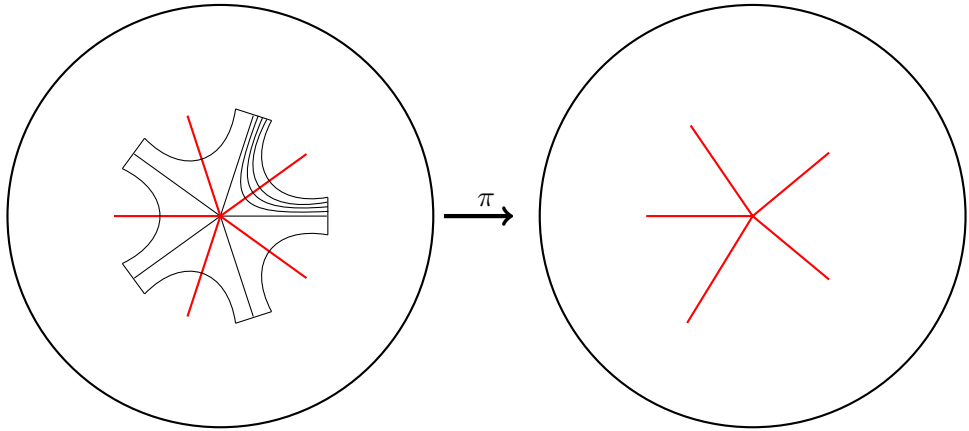
Consider a measured foliation $(\mathcal{F},\mu)$ on $(S,g)$.
It lifts  to a $\pi_1(S)$-equivariant measured foliation $(\mathcal{\widetilde{F}},\widetilde{\mu})$ on the universal cover $(\widetilde{S},g)$.
Let $T$ be the leaf space of $\mathcal{\widetilde{F}}$. There is a natural projection $\pi:\widetilde{S}\rightarrow T$
given by projecting every leaf of $\widetilde{\mathcal{F}}$ to a point. See Figure \ref{fig:leaf}.
We can define a metric $\rho$ on $T$ by pushing forward the measure $\widetilde{\mu}$ by the projection $\pi$, that is,  $\rho=\pi_*\widetilde{\mu}$.
In this way, $(T,\rho)$ becomes a $\mathbb{R}$-tree.
Note that the definition of $(T,\rho)$ only depends on $(\mathcal{F},\mu)$ and the choice of
lifting. The fundamental group $\pi_1(S)$ acts by isometries on $(T,\rho)$ and the map $\pi$ is equivariant with respect to this action.

With the above terminology, now we may state the following result of Wolf \cite{Wolf95, Wolf96}, in a form that we need in this paper.

\begin{proposition}[Proposition 3.1,  \cite{Wolf96}]\label{thm:harmonic}
There is a $\pi_1(S)$-equivariant map $\omega:(\widetilde{S},g)\rightarrow (T,\rho)$ which is equivariantly homotopic to $\pi:(\widetilde{S},g)\rightarrow (T,\rho)$. Off a discrete set, $\omega$ is locally a harmonic projection to a Euclidean line.
Moreover, the  Hopf differential of $\omega$, defined by $\Phi=\omega^*(\rho)^{2,0}$, is
a holomorphic quadratic differential whose vertical measured foliation is equivalent to $(\mathcal{F},\mu)$.
\end{proposition}
Here, we denote by $\omega^*(\rho)^{2,0}$ the $(2,0)$ part of the directional derivatives of $\omega$.
Recall  that the equivalent harmonic map $\omega$ belongs to the Sobolev class $W^{1,2}$
(in the sense of Korevaar and Schoen)
and the directional derivatives of $\omega$ give a symmetric $L^1$ tensor, which can be decomposed to
(according to the conformal structure $g$)
$$\omega^*(\rho)^{2,0}+\omega^*(\rho)^{1,1}+ \omega^*(\rho)^{0,2}.$$

In this paper, the map $\omega$ in Proposition \ref{thm:harmonic} is called an \emph{equivariant harmonic map}  from $(\widetilde{S},g)$  to $(T,\rho)$  or, for simplicity, an equivariant harmonic map  from $(S,g)$  to the $\mathbb{R}$-tree associated with $(\mathcal{F},\mu)$.

\begin{definition} The energy of $\omega$ is defined by
$$E(\omega,g)=\frac{1}{2}\int_S | \omega^*(\rho)^{1,1}|,$$
where the integral domain $S$ is considered as a fundamental domain of $\pi_1(S)$ on the universal cover $\widetilde{S}$.
\end{definition}
Since the harmonic map is $\pi_1(S)$-equivariant, the energy is well-defined.

 \begin{proposition}
The extremal length $\mathrm{Ext}_{\mathcal{F}}(g)$ of the measured foliation $(\mathcal{F},\mu)$ on $(S,g)$ is realized as the energy of
the harmonic map $\omega$.
\end{proposition}

 \begin{proof}
 When the image of $\omega$ is contained in an embedded interval of $(T,\rho)$,
 $\omega$ is similar to a smooth function. Thus we can express
 the energy density on
 the integral domain (almost everywhere) as
 $$| \omega^*(\rho)^{1,1}(z)|=\rho(\omega(z)) \left( |\omega_z|^2+|\omega_{\bar z}|^2\right) dz d\bar z.$$
 By the same reason, the Hopf differential of $\omega$ is given by
 $$\Phi(z)=\rho(\omega(z))\omega_z\overline\omega_z dz d\bar z.$$
Then the proposition follows from the following equations:
 \begin{eqnarray*}
E(\omega,g)&=&\frac{1}{2}\int_S \rho(\omega(z)) \left( |\omega_z|^2+|\omega_{\bar z}|^2\right) \ dzd\overline{z} \\
&=&\int_S \rho(\omega(z))|\omega_z|^2 \ dzd\overline{z}\\
&=&\int_S|\Phi| \\
&=&\mathrm{Ext}_\mathcal{F}(g).
\end{eqnarray*}
The second equality holds since the Jacobian $|\omega_z|^2-|\omega_{\overline{z}}|^2=0$ almost everywhere.
The last two equalities hold since the Hopf differential $\Phi$ of $\omega$ is
equal to the Hubbard-Masur differential $\Phi(\mathcal{F})$ and,  by Proposition \ref{pro:area}, the area of $\Phi(\mathcal{F})$ is equal to
$\mathrm{Ext}_\mathcal{F}(g) $.
\end{proof}

\section{Second variation of extremal length}\label{sec:second}
Let $(S,g_t)$ be a smooth family of hyperbolic metrics on $S$.
Denote $g_t$ by $g_t(z_t) |d z_t|^2$ in the conformal coordinates $z_t$ of $(S,g_t)$.
For simplicity, we shall denote by $g=g_0$ and $z=z_0$.

\subsection{Quasiconformal deformation of  conformal structures}
The identity map between $(S,g)$ and $(S,g_t)$ induces a quasiconformal mapping $z\mapsto z_t$.
We always assume that the Beltrami differentials $\mu(t)=\frac{\partial z_t}{\partial \bar z}/ \frac{\partial z_t}{\partial z}$
satisfy
$$\mu(t)=t\mu +
o(t),$$
where $\|\mu\|_\infty$ is finite.
We may consider $(S,g_t)$ or $\mu(t)$ as a path in $\mathcal{T}(S)$, with tangent vector $\mu$ at the basis point $(S,g)$.

Let $\mathbb{H}^2=\widetilde{S}$ and let $\Gamma$ be the Fuchsian group of $(S,g)$, that is, $(S,g)=\mathbb{H}^2/\Gamma$.
Any Beltrami differential $\mu$ on $(S,g)$ lifts to an automorphic form $\widetilde{\mu}(z)$  on $\mathbb{H}^2$, a.e.,
$$\widetilde{\mu}(h(z))\overline{h'(z)}/h'(z)=\widetilde{\mu}(z)$$
for all $h\in \Gamma$.

Assume that $\mu(t)=\frac{\partial z_t}{\partial \bar z}/ \frac{\partial z_t}{\partial z}=t\mu+o(t)$.
We can lift the quasiconformal map $z\to z_t$ to $\mathbb{H}^2$
such that it satisfies the  Beltrami equation
$$f^t_{\bar z}(z)=\widetilde{\mu}(t)f^t_z(z),$$
with normalized conditions
$f^t(0)=0, f^t(1)=1, f^t(\infty)=\infty$.

Denote by $$V(z)=\lim_{t\to 0}\frac{f^t(z)-z}{t}.$$

The following lemma is well-known. See Ahlfors \cite{Ahlfors}.

\begin{lemma}\label{lem:Ahlfors}
 $\frac{\partial}{\partial \bar{z}} V(z)=\widetilde{\mu}(z)$ holds in
the sense of distribution.
\end{lemma}
In the following, we shall denote the restriction of $V(z)$ on a fundamental domain of $(S,g)$ by $\dot z$.

\subsection{Weil-Petersson geometry}
The Teichm\"uller space $\mathcal{T}(S)$ is a complex manifold which can be endowed with a K\"ahler metric, the Weil-Petersson metric.
Recall that tangent vectors to $\mathcal{T}(S)$ at a point $(S,g)$ can be represented by harmonic Beltrami differentials of the form $\mu=\bar\Phi/g$, where $\Phi$ is a holomorphic quadratic differential on $(S,g)$.
The Weil-Petersson Riemannian inner product of two such tangent vectors is the $L^2$ inner product
$$<\frac{\bar\Phi}{g},\frac{\bar\Psi}{g}>_{\mathrm{WP}}=\mathrm{Re} \int_S \frac{\Phi}{g}\frac{\bar\Psi}{g} d \mathrm{Area}_g=
\mathrm{Re }\int_S \frac{\Phi(z)\overline{\Psi(z)}}{g(z)} |d z|^2.$$
Although the Weil-Petersson metric is not complete,  it is geodesically convex \cite{Wolpert87}. This means that any two points in $\mathcal{T}(S)$ can be joined by a unique Weil-Petersson geodesic.

If the family of hyperbolic metrics $(S,g_t)$ agree through second order at $(S,g)$ (when $t=0$) with a Weil-Petersson geodesic, then we say that
$(S,g_t)$ is  \emph{Weil-Petersson geodesic} at $(S,g)$.
Although the equation for a family of hyperbolic metrics $(S,g_t)$ to be a Weil-Petersson geodesic is unknown,
 by Ahlfors  \cite{Ahlfors}, $\mu(t)=t\frac{\overline{\Phi}}{g}$ is Weil-Petersson geodesic at $t=0$.

Since we mainly consider the second variation along Weil-Petersson geodesics, in the following discussion, we always assume that $\mu(t)=t\mu +
o(t)$ satisfying
$\mu=\frac{\overline{\Phi}}{g}$ and
$\ddot{\mu}=\frac{d^2}{dt^2}\mu(t)|_{t=0}\equiv 0$.

\begin{remark}
We may assume that $(S,g_t)$ is given by

$$g_t=t \Phi dz^2+ g \left( \mathcal{H}(t)+\frac{t^2|\Phi|^2}{g^2\mathcal{H}(t)} \right)dz d\bar z+ t\bar\Phi d \bar{z}^2$$
where $\mathcal{H}(t)$ is the solution of the Bochner equation

$$\Delta_g \log \mathcal{H}(t)=2\mathcal{H}(t)- 2\frac{t^2|\Phi|^2}{g^2\mathcal{H}(t)}-2.$$
In this situation, the identity map $\mathrm{id}: (S,g) \to (S,g_t)$ is harmonic and $\mu(t)$ is Weil-Petersson
geodesic at $t=0$. See  Wolf \cite{Wolf89}.
\end{remark}

\subsection{Variations of energy (extremal length)}
From now on, we shall identify extremal length as energy of equivariant harmonic maps.
We fix a simple closed curve $\gamma$ on $S$ and let $(T,\rho)$ be the $\mathbb{R}$-tree associated with $\gamma$.
The family of hyperbolic metrics $(S,g_t)$ will be a Weil-Petersson geodesic. Let
$\omega^t:(\widetilde{S},g_t)\rightarrow (T,\rho)$ be the corresponding
harmonic maps.
For simplicity, we denote  $\omega=\omega^0$.

In this part and \S \ref{sec:field}, we investigate the second variational formula of $E(\omega^t, g_t)=\mathrm{Ext}_\gamma(g_t)$.
One may wishes to extend the formula to the extremal
length function of any measured foliation $\mathcal{F}$. However, one of the technical difficulties is the
regularity of $\omega^t$ with respect to $t$. Even in the case of a simple closed curve $\gamma$,
the extremal length function $\mathrm{Ext}_{\gamma}(\cdot)$ may not be  $C^2$ on $\mathcal{T}(S)$.
 We shall discuss the smoothness of  $\mathrm{Ext}_{\gamma}(\cdot)$ on \S \ref{sec:regularity}.

Another possible generalization is to study the variation of $\mathrm{Ext}_{\mathcal{F}_t}(g_t)$ where $\mathcal{F}_t$ is varied
(such a situation appears when one study the Teichm\"uller distance function between a Teichm\"uller geodesic and a fixed point).

There are some technical results which we will use throughout this paper:
\begin{enumerate}
\item
As we have mentioned above, in this case, the complement of the critical leaves  of the measured foliation
equivalent to $\gamma$ is homeomorphic to a cylinder $C_\gamma$.
If we look at the projection of the harmonic map $\omega:(\widetilde{S},g)\to (T,\rho)$ on $C_\gamma$, each closed leaf (homotopic to
$\gamma$) is mapped to a point and the image of $C_\gamma$ is isometric to the interval $[0,1]$ (see \cite{Wolf95}).

Now suppose that $\mathrm{Ext}_\gamma(\cdot)$ is smooth in an open neighborhood of $(S,g_0)$.
When we deform the hyperbolic structure $(S,g_0)$ into $(S,g_t)$, the corresponding
conformal structures $z_t$, when restricted on $C_\gamma$,  vary smoothly. This is because the conformal structure on a cylinder domain is
 characterized by its conformal modulus, which is equal the reciprocal of the extremal length of the core curve.
Then for all $t$ sufficiently small, the harmonic maps $\omega^t: (C_\gamma, z_t)\to [0,1]$ satisfy the same boundary condition,
 that is, $\omega^t$ maps the boundary of $C_\gamma$ to $\{0,1\}$.
It follows from  Eells-Lemaire \cite[\S 4]{EL} that  $\omega^t$ depends smoothly on $t$.

\item Note that  we can consider $\omega$ as a real-valued harmonic function
(since we are only interested in the integration over a fundamental domain of $S$, we can consider
$\omega$ as a harmonic function on the cylinder $C_\gamma$, which is a full measure subset of $S$).
In this case, $\omega$ satisfies the equation $\omega_{z \bar z}=0$ and the Jacobian
$|\omega_z|^2-|\omega_{\bar z}|^2=0.$
\end{enumerate}

In the remaining part of this section, we assume that the family of harmonic maps varies smoothly on $t$.

\begin{definition}
The \emph{variational vector} of $\omega^t$ at $t=0$ is defined by
$$\dot\omega(z)=\lim_{t\to 0}\frac{\omega^t(z)-\omega(z)}{t}$$
 locally under conformal coordinates of $(S,g)$.
\end{definition}

Denote by
$E(\omega^t,g_t)$   the energy of  $\omega^t$, which is  equal to the extremal length of $\gamma$ on $(S,g_t)$. That is,
$$E(\omega^t,g_t)=\int_S|\frac{\partial \omega^t}{\partial z_t}|^2 \ dz_td\bar{z_t}.$$

To study the second variation,
we  first separate the overall variation into a term that refers only to the second variation of the metrics $g_t$ and another term
that refers only to the second variation of the maps $\omega^t$. Such a separation is quite standard for a variational functional (see Wolf \cite{Wolf09}
for example).

Formally, we set $$g_t=g+t \dot g+o(t)$$ and
$$\omega^t=\omega+ t \dot\omega+\frac{t^2}{2}\ddot\omega +o(t^2).$$
One can consider $\dot g$ as a tangent vector to $\mathcal{M}_{-1}$ at $g_0$ and $\dot\omega$
as a tangent vector to $C^{\infty}(\widetilde S)$ at $\omega^0$. In the following, we denote by
$D_1 E[\dot \omega]$ the directional derivative of $E$ in the direction $\dot\omega$
and $D_2 E$ the directional derivative of $E$ in the direction $\dot g$. The directional
derivative of $D_2 E[\dot g]$ in the direction $\dot g$ is denoted by $D^2_{22}E[\dot g, \dot g]$
and vice versa.

\begin{proposition}[Separation of second variation]
$$\frac{d^2}{dt^2}|_{t=0}E(\omega^t,g_t)=-D^2_{11}E[\dot \omega,\dot\omega]+D^2_{22}E[\dot g,\dot g].$$
\end{proposition}
\begin{proof}
Since the energy of a harmonic map is stationary with respect to
variation of the map, we have
$D_1 E(\omega^t,g_t)[\dot
\omega]=0$
and $D_1 E(\omega^t,g_t)[\ddot \omega]=0$.

It follows that
\begin{eqnarray*}
0 &=&\frac{d}{dt}|_{t=0}D_1 E(\omega^t,g_t)[\dot \omega] \\
&=&D^2_{11}E[\dot \omega,\dot\omega]+D^2_{12}E[\dot \omega,\dot g]+D_1 E[\ddot \omega] \\
&=& D^2_{11}E[\dot \omega,\dot\omega]+D^2_{12}E[\dot \omega,\dot g].
\end{eqnarray*}
As a result,
$$D^2_{11}E[\dot \omega,\dot\omega]=-D^2_{12}E[\dot \omega,\dot g].$$
The second variation of extremal length function is given by
$$
\frac{d^2}{dt^2}|_{t=0} E(\omega^t,g_t)
= D^2_{11}E[\dot \omega,\dot\omega]+2D^2_{12}E[\dot \omega,\dot g] +
D^2_{22}E[\dot g,\dot g].$$
Thus, we have
\begin{equation}\label{equ:sec}
\frac{d^2}{dt^2}|_{t=0}E=-D^2_{11}E[\dot \omega,\dot\omega]+D^2_{22}E[\dot g,\dot g].
\end{equation}
\end{proof}

 The terms $D^2_{22}E[\dot g,\dot g]$ and $D^2_{11}E[\dot \omega,\dot\omega]$ are formulated in the following lemmas.

 \begin{lemma} With the above notation, we have
$$ D^2_{22}E [\dot g,\dot g]=\frac{d^2}{dt^2}|_{t=0}E(\omega,g_t)=2\int_S |\mu|^2(|\omega_z|^2+|\omega_{\bar z}|^2) \ dzd\bar z.$$
 \end{lemma}
\begin{proof}
 By definition,
 $$E(\omega^t,g_t)=\int_S|\frac{\partial \omega^t}{\partial z_t}|^2 \ dz_td\bar{z_t}.$$
 Note that
 $$dz_t=(z_t)_z dz+(z_t)_{\bar z} d \bar z ,$$
 $$d \overline z_t=(\bar z_t)_z dz+(\bar z_t)_{\bar z} d \bar z,$$
 $$\mu(t)=\frac{(z_t)_{\bar z}}{(z_t)_z},$$
\begin{equation}\label{equ3}
dz_td\bar{z_t}=|(z_t)_z|^2(1-|\mu(t)|^2)dz d \bar z.\end{equation}

 It follows from the chain rule of differential that

\begin{eqnarray*}
\omega_{z_t}&=&\frac{\omega_z(\bar{z_t})_{\bar z}-\omega_{\bar z}(\bar{z_t})_z}{|(z_t)_z|^2(1-|\mu(t)|^2)} \\
&=&\frac{\omega_z-\overline{\mu(t)}\omega_{\overline{z}}}{(z_t)_z(1-|\mu(t)|^2)}.
\end{eqnarray*}

As a result,

\begin{eqnarray}\label{equ4}
|\omega_{z_t}|^2&=&\frac{(\omega_z-\overline{\mu(t)}\omega_{\bar{z}})\overline{(\omega_z-\overline{\mu(t)}\omega_{\bar{z}})}}
{|(z_t)_z|^2(1-|\mu(t)|^2)^2} \\
&=&\frac{|\omega_z|^2+|\mu(t)|^2|\omega_{\bar z}|^2-2 \mathrm{Re}(\mu(t)\omega_z\overline \omega_z)}{|(z_t)_z|^2(1-|\mu(t)|^2)^2}.\nonumber
\end{eqnarray}

Combining equation $(\ref{equ3})$ with $(\ref{equ4})$, we have
$$E(\omega,g_t)=\int_S \frac{|\omega_z|^2+|\mu(t)|^2|\omega_{\bar z}|^2-2 \mathrm{Re}(\mu(t)\omega_z\overline \omega_z)}{1-|\mu(t)|^2} d z d\bar z.$$

By our assumption,  $\mu(t)=t\mu+o(t^2)$. Thus we have
$$E(\omega,g_t)=\int_S \frac{|\omega_z|^2+t^2|\mu|^2|\omega_{\bar z}|^2-t2 \mathrm{Re}(\mu\omega_z\overline \omega_z)}{1-t^2|\mu|^2}d z d\bar z+o(t^2).$$

Now we consider the variation of $E(\omega,g_t)$. Note that

\begin{eqnarray*}
\frac{d}{dt}E(\omega,g_t)&=&\int_S \frac{2t|\mu|^2|\omega_{\overline z}|^2-2 \mathrm{Re}(\mu\omega_z\overline \omega_z)}{1-t^2|\mu|^2}d z d\bar z \\
&& -\int_S (|\omega_z|^2+t^2|\mu|^2|\omega_{\bar z}|^2-t2 \mathrm{Re}(\mu\omega_z\overline \omega_z))\frac{-2t|\mu|^2}{|1-t^2|\mu|^2|^2}d z d\bar z \\
&& + o(t).
\end{eqnarray*}
It follows from direct calculation that
\begin{equation}\label{equ5}
\frac{d^2}{dt^2}|_{t=0}E(\omega,g_t)=2\int_S |\mu|^2(|\omega_z|^2+|\omega_{\bar z}|^2) \ dzd\bar z.
\end{equation}
\end{proof}

\begin{remark}\label{remark:Gardiner}
The above proof shows that
$$\frac{d}{dt}|_{t=0}E(\omega,g_t)=-2\mathrm{Re}\int_S \mu\omega_z \overline{\omega}_z dzd\bar{z}.$$
Since
$$\frac{d}{dt}|_{t=0}E(\omega^t,g_t)= D_1 E[\dot\omega]+ D_2 E[\dot g]$$
 and
 $$D_1 E[\dot
\omega]=0,
$$
we obtain the following first variational formula:

\begin{equation}
\frac{d}{dt}|_{t=0}E(\omega^t,g_t)=-2\mathrm{Re}\int_S\Phi \mu
\end{equation}
 where $\Phi=\omega_z
\overline{\omega}_zdz^2\in QD(g)$, which is the Hubbard-Masur
differential for $\gamma$ at $(S,g)$.
Actually, the above formula is valid for any measured foliation (see Gardiner \cite{Gardiner} and Wentworth  \cite{Wentworth}).

\end{remark}

The next lemma is to evaluate the term $D^2_{22}E[\dot\omega,\dot\omega]=\frac{d^2}{dt^2}|_{t=0}E(\omega^t,g)$ in $(\ref{equ:sec})$.
\begin{lemma} We have
$$\frac {d^2} {dt^2}|_{t=0} E(\omega^t, g)=2 \int_S | \dot \omega_z|^2 \ dz d\bar{z}.$$
\end{lemma}
\begin{proof}
As before, we set $$\omega^t=\omega+t\dot\omega+\frac{t^2}{2} \ddot\omega+o(t^2).$$ Then

$$E(\omega^t, g)=\int_S |\omega_z^t|^2 \ dzd \bar z$$
and
\begin{eqnarray*}
|\frac{\partial \omega^t}{\partial z}|^2&=&(\omega_z+t\dot\omega_z+\frac{t^2}{2}\ddot\omega_z+o(t^2))\overline{(\omega_z+t\dot\omega_z+\frac{t^2}{2}\ddot\omega_z+o(t^2))} \\
&=&|\omega_z|^2+2t \ \mathrm{Re}(\dot\omega_z \overline{\omega}_{\bar{z}})+ t^2(|\dot\omega_z|^2+\mathrm{Re}(\omega\overline{\ddot\omega_z}_{\bar{z}}))+o(t^2).
\end{eqnarray*}

Then
\begin{eqnarray*}
\frac d {dt}|_{t=0}E(\omega^t, g)
&=&2\mathrm{Re}\int_S \dot\omega_z \overline{\omega}_{\bar{z}} \ dzd\bar z  \\
( \mathrm{integration} \ \mathrm{by} \ \mathrm{parts})  &=&-2\mathrm{Re }\int_S \dot \omega \overline {\omega_{z \bar z}} \ dz d \bar z \\
(\mathrm{by} \ \mathrm{harmonicity})&=&0.
\end{eqnarray*}

And then

\begin{eqnarray}\label{equ6}
\frac {d^2} {dt^2}|_{t=0} E(\omega^t, g)&=&
2 \int_S | \dot \omega_z|^2 dz d\bar{z}+2\mathrm{Re}\int_S \omega_z\overline{\ddot\omega}_{\bar{z}} \ dzd\overline{z} \\
( \mathrm{integration} \ \mathrm{by} \ \mathrm{parts})&=&2 \int_S | \dot \omega_z|^2 dz d\bar{z}-2\mathrm{Re}\int_S \omega_{z\bar{z}}\overline{\ddot\omega} \ dzd\bar{z} \nonumber \\
(\mathrm{since} \ \omega_{z\bar{z}}=0) &=& 2 \int_S | \dot \omega_z|^2 dz d\bar{z}. \nonumber
\end{eqnarray}
\end{proof}
Combining $(\ref{equ:sec}), (\ref{equ5})$ and $(\ref{equ6})$, we have
\begin{eqnarray*}\label{equ7}
\frac{d^2} {dt^2}|_{t=0}E(\omega^t,g^t)&=& 2 \int_S
|\mu|^2(|\omega_z|^2+|\omega_{\bar z}|^2) \ dzd \bar z -2\int_S |\dot
\omega_z|^2 \ dzd \bar z.\nonumber
\end{eqnarray*}

In conclusion, we have the following:
\begin{theorem}\label{theorem:second}
Let $\gamma$ be a simple closed curve on $S$ and let $g_t$ be a Weil-Petersson geodesic in
$\mathcal{T}(S)$ with Beltrami differential $\mu(t)=t\mu+o(t^2)$. Suppose that $\mathrm{Ext}_{\gamma}(g_t)$ is a smooth function of $t$.
Denote by
$\omega$ is the harmonic map from $(\widetilde{S},g_0)$ to the $\mathbb{R}$-tree determined by $\gamma$. Then

\begin{eqnarray*}
\frac{d^2} {dt^2}|_{t=0}\mathrm{Ext}_{\gamma}(g_t)&=& 2\int_S
|\mu|^2(|\omega_z|^2+|\omega_{\bar z}|^2) \ dzd \bar z -2\int_S |\dot
\omega_z|^2 \ dzd \bar z \\
&=& 4 \int_S
|\mu|^2|\omega_z|^2 \ dzd \bar z -2\int_S |\dot
\omega_z|^2 \ dzd \bar z .
\end{eqnarray*}
\end{theorem}

\subsection{The variational vector $\dot\omega$}\label{sec:field}

We have shown that $$\frac {d^2} {dt^2}|_{t=0} E(\omega^t, g_0)=2 \int_S | \dot \omega_z|^2 dz d\bar{z}.$$
In order to compare this formula with
$(\ref{equ5})$, it is important to find an expression for $\dot\omega$ or $\dot\omega_z$ in terms of $\mu$ and $\omega$.
Note that  we can consider $\omega$ as a real-valued harmonic function.
In this case, $(\ref{equ5})$ becomes
\begin{equation*}
\frac {d^2} {dt^2}|_{t=0} E(\omega, g_t)=4 \int_S |\mu|^2|\Phi(z)| dz d\bar{z},
\end{equation*}
where $\Phi(z) dz^2=\omega_z\overline \omega_zd z^2$ is the Hopf differential of $\omega$. It will be interesting to
give an expression of $\dot\omega$ or $\dot\omega_z$ in term of $\mu$ and $\Phi$.

\bigskip

Consider the harmonic map equation $$H(\omega^t,z_t)=\frac{\partial^2\omega^t}{\partial z_t\partial \overline{z_t}}=0.$$
Differentiating in $t$, we have
\begin{equation}\label{equ8}
-\frac{d}{dt}|_{t=0}H(\omega,z_t)=\frac{d}{dt}|_{t=0}H(\omega^t,z)=\dot\omega_{z\overline{z}}.
\end{equation}

To compute $\frac{d}{dt}|_{t=0}H(\omega,z_t)$, we express the operator $\frac{\partial^2}{\partial z_t \partial \bar{z_t}}$ as
$$\{\frac 1 {1-t^2 | \mu|^2} \frac 1 {(z_t)_z}(\partial_z -t
 \overline {\mu}\partial_{\bar{z}})\} \circ
\{\frac 1 {1-t^2 |\mu|^2} \frac 1 {\overline {(z_t)_z}} (-t
\mu\partial_z+\partial_{\bar z})\}+o(t)$$
$$=\{\frac 1 {(z_t)_z}(\partial_z -t
\overline {\mu}\partial_{\bar z})\} \circ \{ \frac 1 {\overline {(z_t)_z}}
(-t \mu\partial_z+\partial_{\bar z})\}+o(t).$$
It is not hard to verify that
\begin{eqnarray*}
H(\omega,z_t)&=& -t(\mu_z\omega_z+\mu\omega_{zz})+ (-t\mu\omega_z+\omega_{\bar z}) (-\overline{(z_t)_{z\bar z}}) \\
&& -t\overline{\mu}( \omega_{\bar z\bar z}+ \overline{(z_t)_{z\bar z}}\omega_{\bar z})+o(t).
\end{eqnarray*}
As a result,
\begin{eqnarray*}
-\frac{d}{dt}|_{t=0}H(\omega,z_t)&=& \mu_z\omega_z+\mu\omega_{zz}+ \omega_{\bar z}\overline{\dot z_{z \bar z}}+\overline{\mu} \omega_{\bar z\bar z}.
\end{eqnarray*}

By Lemma \ref{lem:Ahlfors}, $\mu=\dot z_{\overline z}$, hence $\overline{\dot z_{z\overline{z}}}=\overline {\mu_z}$, and then
\begin{eqnarray}\label{equ9}
-\frac{d}{dt}|_{t=0}H(\omega,z_t)&=&\mu \omega_{zz}+\overline{ \mu}\omega_{\bar z \bar z}
+\overline {\mu_z} \omega_{\bar z}+\mu_z \omega_z  \nonumber \\
&=&\frac{\partial}{\partial z}(\mu\omega_z)+\frac{\partial}{\partial \overline{z}}(\overline{\mu}\omega_{\overline{z}}).
\end{eqnarray}
It follows from $(\ref{equ8})$ and $(\ref{equ9})$ that
\begin{equation}\label{equ10}
\dot\omega_{z\bar z}=\frac{\partial}{\partial z}(\mu\omega_z)+\frac{\partial}{\partial \bar{z}}(\overline{\mu}\omega_{\bar{z}}).
\end{equation}

Using  $(\ref{equ10})$ and integration by part, we have
\begin{eqnarray}\label{equ11}
\int_S |\dot \omega_z|^2 \ dzd \bar z&=&
-\int_S \overline{\dot \omega}\dot \omega_{z \bar z} \ dz d\bar{z} \nonumber \\
&=&-\int_S \frac{\partial}{\partial z}(\mu\omega_z)\overline{\dot \omega}+\frac{\partial}{\partial \overline{z}}(\overline{\mu}\omega_{\bar{z}})\overline{\dot \omega} \  dzd\bar{z} \nonumber \\
&=& \int_S \mu\omega_z\overline{\dot\omega}_z+\overline{\mu}\omega_{\overline{z}}\overline{\dot\omega}_{\bar{z}} \ dzd\bar{z}.
\end{eqnarray}

The equality $(\ref{equ11})$ will be used in \S \ref{sec:pluri} to Theorem \ref{theorem:WP}.

\section{On the regularity of extremal length}\label{sec:regularity}
In this section, we discuss some partial results on the regularity of the extremal length functions.

Given a measured foliation $\mathcal{F}$ on $S$, the extremal length $\mathrm{Ext}_{\mathcal{F}}(\cdot)$ is a $C^1$-function on
$\mathcal{T}(S)$, see \cite[Proposition 4.2]{GM}.

Let $Q(S)$ be the cotangent bundle of $\mathcal{T}(S)$, with $\pi$ the natural projection.
Due to Hubbard-Masur \cite{HM},
there is a homeomorphism
$$Q(S) \to \mathcal{T}(S)\times \mathcal{MF},$$
which associates each holomorphic quadratic differential $\Phi$ with its underlining complex structure $\pi(\Phi)$ and its vertical measured foliation $v(\Phi)$. Fix a measured foliation $\mathcal{F}$ on $S$, let
$$E(\mathcal{F})=\{\Psi\in Q(S) : v(\Psi)=\mathcal{F}\}.$$
It follows that $\pi: E(\mathcal{F})\to \mathcal{T}(S)$ is a homeomorphism. It is the inverse of the Hubbard-Masur map, that is,
$$\pi \circ \Phi_{\mathcal{F}}(X) = X, $$
where $\Phi_{\mathcal{F}}(X)$ denotes the holomorphic quadratic differential on $X$ with vertical measured foliation equivalent to $\mathcal{F}$.

With the above notation, it follows from Proposition \ref{pro:area} that the extremal length $\mathrm{Ext}_{\mathcal{F}}(\cdot)$
is given by the norm $\|\Phi_{\mathcal{F}}(\cdot)\|$.
It was observed by Royden \cite{Royden} that when $\Phi_{\mathcal{F}}(X)$ has zeros of order at least $3$,
$\|\Phi_{\mathcal{F}}(\cdot)\|$ is not $C^2$. Note that a Jenkins-Strebel differential realizing some simple closed curve can have higher-order
zeros and, if this happen, it is unreasonable to compute the second derivative of the extremal length.

To avoid such a difficulty, we apply a result of Masur \cite{Masur}. Denote by $\mathcal{Q}_1$ the subset of quadratic differentials on
$S$ with simple zeros (in the study of Teichm\"uller flows, $\mathcal{Q}_1$ is known as the \emph{principal stratum}).
 $\mathcal{Q}_1$ is an open and dense subset
of $Q(S)$. Masur  \cite{Masur} showed that when $\Phi_0\in \mathcal{Q}_1$, there exist real analytic coordinates for $\mathcal{Q}_1$ near $\Phi_0$ such that the
norm $\|\Phi\|$ is real analytic.\footnote{The point is that, each $\Phi$ near $\Phi_0$ has a geodesic triangulation (the \emph{Delaunay triangulation}) such that the vertices are zeros of $\Phi$. The real analytic coordinates are given by the (directed) side lengths of the triangles. The norm $\|\Phi\|$ is just the sum of areas of these triangles, which are obviously real analytic functions of the side lengths, as soon as all the triangles don't collapse at $\Phi_0$.
} If, moreover, $\Phi_0\in \mathcal{Q}_1\cap E(\mathcal{F})$, then the restriction of the analytic coordinates
on $\mathcal{Q}_1\cap E(\mathcal{F})$ implies that the map
$\pi:  E(\mathcal{F}) \to \mathcal{T}(S)$
is  real analytic in a neighhorhood of $\Phi_0$. As a result, for any $(X,\mathcal{F})\in \mathcal{T}(S)\times \mathcal{MF}$
such that $\Phi_{\mathcal{F}}(X)$ has only simple zeros, the extremal length function $\mathrm{Ext}_{\mathcal{F}}(\cdot)$
is real analytic in a neighborhood of $X$.

In our study, we mainly consider measured foliations corresponding to simple closed curves.
Let $\gamma$ be a simple closed curve on $S$, denote
$$\mathcal{T}_\gamma= \{X\in \mathcal{T}(S) :  \Phi_{\gamma}(X)\in \mathcal{Q}_1\},$$
i.e., the subset of Teichm\"uller space such that the holomorphic quadratic differentials realizing $\gamma$ have only simple zeros.

\begin{lemma}\label{lem:denseopen}
The subset $\mathcal{T}_\gamma$ is open and dense in $\mathcal{T}(S)$.
\end{lemma}
\begin{proof}Although the result may be known to the experts,
 we include a proof for completeness. It is clear that $\mathcal{T}_\gamma$ is an open subset of $\mathcal{T}(S)$.
Note that $\mathcal{T}_\gamma$  is also dense: by an argument of Hubbard-Masur, at a multiple zero of a quadratic differential,
we can break up the zeros into simple zeros and keep the same vertical foliation.

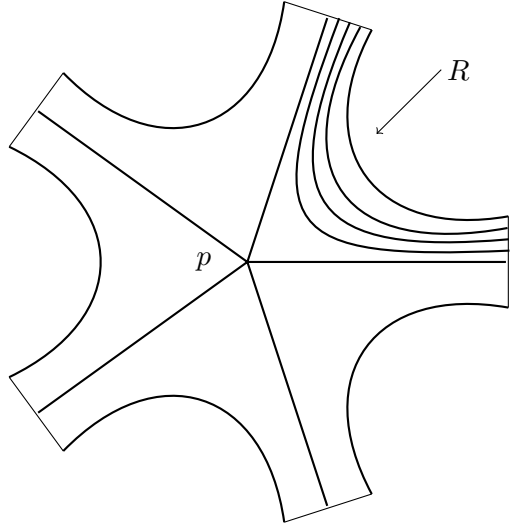
\begin{figure}[htbp]
\centering
\begin{tikzpicture}[scale=1.7]
\path (0,0) coordinate (origin);
\path (0:2cm) coordinate (P0);
\path (1*72:2cm) coordinate (P1);
\path (2*72:2cm) coordinate (P2);
\path (3*72:2cm) coordinate (P3);
\path (4*72:2cm) coordinate (P4);
\path (10:2.05cm) coordinate (V1);
\path (62:2.05cm) coordinate (V2);
\path (82:2.05cm) coordinate (V3);
\path (134:2.05cm) coordinate (V4);
\path (154:2.05cm) coordinate (V5);
\path (206:2.05cm) coordinate (V6);
\path (226:2.05cm) coordinate (V7);
\path (-82:2.05cm) coordinate (V8);
\path (-62:2.05cm) coordinate (V9);
\path (-10:2.05cm) coordinate (V10);
\draw [thick] (origin) -- (P0) (origin) -- (P1) (origin) -- (P2)
(origin) -- (P3) (origin) -- (P4);
\draw (10:2.05cm) -- (-10:2.05cm);
\draw  (82:2.05cm) -- (62:2.05cm);
\draw  (154:2.05cm) -- (134:2.05cm);
\draw  (226:2.05cm) -- (206:2.05cm);
\draw  (-62:2.05cm) -- (-82:2.05cm);
\draw [thick] (V1) .. controls  (10:1cm) and (62:1cm) .. (V2);
\draw [thick] (V3) .. controls  (82:1cm) and (134:1cm) .. (V4);
\draw [thick] (V5) .. controls  (154:1cm) and (206:1cm) .. (V6);
\draw [thick] (V7) .. controls  (226:1cm) and (-82:1cm) .. (V8);
\draw [thick] (V9) .. controls  (-62:1cm) and (-10:1cm) .. (V10);
\draw [thick] (5:2.03cm) .. controls  (5:0.5cm) and (67:0.5cm) .. (67:2.03cm);
\draw [thick] (7.5:2.03cm) .. controls  (2.5:0.75cm) and (69.5:0.75cm) .. (64.5:2.03cm);
\draw [thick] (2.5:2.03cm) .. controls  (2.5:0.25cm) and (69.5:0.25cm) .. (69.5:2.03cm);
\draw [->] (1.5,1.5) -- (1,1);
\node [left] at (1.8,1.5) {$R$};
\node [left] at (-0.2,0) {$p$};
\end{tikzpicture}
\caption{The neighborhood $\mathcal{N}$ of the multiple zero point $p$, obtained by gluing  $n+2$ copy of a rectangle along vertical sides.}
\label{fig:sig}
\end{figure}

To be more precise, let $p$ be an $n$-order zero of a quadratic differential $\Phi$ on the Rieamnn surface $X$.
We choose a  neighhorhood $\mathcal{N}$ of $p$ which is obtained by gluing  $n+2$ copy of a rectangle $R$. As shown on the right of
Figure \ref{fig:sig}, the vertical lines in each rectangle correspond to the vertical leaves of $\Phi$.
We assume that the vertical and horizontal length of $R$  is equal to $a$ and $b$, respectively.

Let $\widetilde{R}$ be another rectangle, with vertical and horizontal length equal to $a+\epsilon$ and $b+\delta$, respectively.
We define $f:R \to \widetilde{R}$ be the affine map, which maps horizontal (vertical) lines to horizontal (vertical) lines.
Note that the quasiconformal
dilatation $K(f)$ is equal to $\frac{1+\epsilon/a}{1+\delta/b}$. In the following, we choose  $\epsilon>0$ and $\delta>0$ sufficiently small and satisfying
$\frac{a+\epsilon}{2b+\delta}=\frac{a}{2b}$.

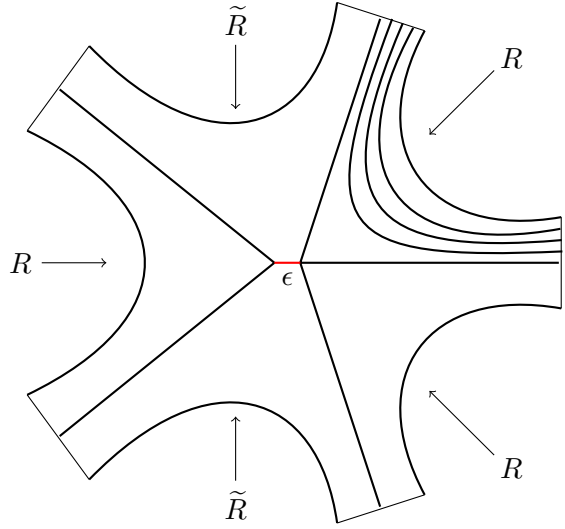
\begin{figure}[htbp]
\centering
\begin{tikzpicture}[scale=1.7]
\path (0,0) coordinate (origin);
\path (-0.2,0) coordinate (Q);
\path (0:2cm) coordinate (P0);
\path (1*72:2cm) coordinate (P1);
\path (2*72:2.3cm) coordinate (P2);
\path (3*72:2.3cm) coordinate (P3);
\path (4*72:2cm) coordinate (P4);
\path (10:2.05cm) coordinate (V1);
\path (62:2.05cm) coordinate (V2);
\path (82:2.05cm) coordinate (V3);
\path (134:2.35cm) coordinate (V4);
\path (154:2.35cm) coordinate (V5);
\path (206:2.35cm) coordinate (V6);
\path (226:2.35cm) coordinate (V7);
\path (-82:2.05cm) coordinate (V8);
\path (-62:2.05cm) coordinate (V9);
\path (-10:2.05cm) coordinate (V10);
\draw [thick] (origin) -- (P0) (origin) -- (P1) (Q) -- (P2) (Q) -- (P3) (origin) -- (P4);
\draw [thick][red]  (Q) -- (origin);
\draw (10:2.05cm) -- (-10:2.05cm);
\draw  (82:2.05cm) -- (62:2.05cm);
\draw  (V4) -- (V5);
\draw  (V6) -- (V7);
\draw  (-62:2.05cm) -- (-82:2.05cm);
\draw [thick] (V1) .. controls  (10:1cm) and (62:1cm) .. (V2);
\draw [thick] (V3) .. controls  (82:1cm) and (134:1cm) .. (V4);
\draw [thick] (V5) .. controls  (154:1cm) and (206:1cm) .. (V6);
\draw [thick] (V7) .. controls  (226:1cm) and (-82:1cm) .. (V8);
\draw [thick] (V9) .. controls  (-62:1cm) and (-10:1cm) .. (V10);
\draw [thick] (5:2.03cm) .. controls  (5:0.5cm) and (67:0.5cm) .. (67:2.03cm);
\draw [thick] (7.5:2.03cm) .. controls  (2.5:0.75cm) and (69.5:0.75cm) .. (64.5:2.03cm);
\draw [thick] (2.5:2.03cm) .. controls  (2.5:0.25cm) and (69.5:0.25cm) .. (69.5:2.03cm);
\draw [->] (1.5,1.5) -- (1,1);
\draw [->] (-0.5, 1.7) -- (-0.5,1.2);
\node [left] at (1.8,1.6) {$R$};
\node [above] at (-0.5, 1.7) {$\widetilde{R}$};
\draw [->] (1.5,-1.5) -- (1,-1);
\node [left] at (1.8,-1.6) {$R$};
\draw [->] (-0.5, -1.7) -- (-0.5,-1.2);
\node [below] at (-0.5, -1.7) {$\widetilde{R}$};
\draw [->] (-2, 0) -- (-1.5,0);
\node [left] at (-2, 0) {$R$};
\node [below] at (-0.1, 0) {$\epsilon$};
\end{tikzpicture}
\caption{The zero $p$ is break up into two lower-order zeros.  The neighhorhood $\widetilde{\mathcal{N}}$ can be identified with  $\mathcal{N}$ up to
scale.}
\label{fig:move}
\end{figure}

The surgery of $\Phi$ in the neighborhood $\mathcal{N}$  is shown in Figure \ref{fig:move}, we can gluing $n$ copy of $R$ and $2$ copy of $\widetilde{R}$ together
along vertical sides
such that the union is a polygon $\widetilde{\mathcal{N}}$. The flat structure on $\widetilde{\mathcal{N}}$ is isometric
to a neighborhood of a quadratic differential. Since $\frac{a+\epsilon}{2b+\delta}=\frac{a}{2b}$,
  $\widetilde{\mathcal{N}}$ can be glue to $X\setminus \mathcal{N}$ using a scaling map $g: \widetilde{\mathcal{N}} \to \mathcal{N}$ (which is holomorphic).
We denote the new flat structure by $\widetilde{\Phi}$. It is a holomorphic quadratic differential on $\widetilde{X}=\widetilde{\mathcal{N}}\sqcup_g (X\setminus \mathcal{N})$ which split the zero $p$ into two distinct lower-order zeros. The vertical foliation of
$\widetilde{\Phi}$ is equivalent to that of $\Phi$ (what we have done is just a Whitehead move).
There is a quasiconformal mapping $F$ between $X$ and $\widetilde{X}$,
which is defined by identity on $X\setminus \mathcal{N}$; and on each copy of $R$, $F$ is defined by $h\circ f$ or $h$. The quasiconformal dilattation
$K(F)$ is the same as $K(f)$. By taking $\epsilon, \delta$ sufficiently small, we can make $K(f)$ close to $1$. This implies that  $\widetilde{X}$
can be taken arbitrary close to $X$ in Teichm\"uller space.

 By doing this surgery inductively, we could obtain a quadratic differential with only
simple zeros, and the quasiconformal dilatation can be controlled very well.
\end{proof}

\begin{remark}
We remark that in some applications, 
the extremal length of simple closed curves that are concerned with are actually real analytic on an appropriate Teichm\"uller space.
For example, the extremal length of any sufficiently symmetric homotopy class of curves on Teichm\"uller of Denjoy domains, studied by Penner \cite{Penner};
and the extremal length some closed curves related to the period problem of Weierstrass representation of minimal surface,
which are real analytic on the Teichm\"uller space of zigzags \cite{WW}. In these examples, the extremal length functions are essentially given by the
Schwarz-Christoffel integrals, whose real analyticity could be seen directly.
\end{remark}

\section{Pluri-subharmonicity of extremal length functions}\label{sec:pluri}

In this section, we prove Theorem \ref{theorem:WP} and Theorem \ref{theorem:main}.
Theorem \ref{theorem:WP} will be restated in the following as Theorem \ref{theorem:sum}.
Theorem \ref{theorem:main} then follows from Theorem \ref{theorem:sum} and the discussion in \S \ref{sec:regularity}, by approximation.

\subsection{Pluri-subharmonicity}
A $C^2$ real-valued function $F$ on a complex manifold $M$ is \emph{(strictly) pluri-subharmonic}  if the \emph{Levi-form} $\partial\overline{\partial}F$ is
(positive) definite at each point of $M$. Recall that $\partial\overline{\partial}F$ is a 2-form defined by
$$\partial\overline{\partial}F=\frac{\partial^2 F}{\partial z^{\alpha}\partial \bar{z}^{\beta}} \ dz^{\alpha} \wedge d\bar{z}^{\beta}$$
in holomorphic coordinates.
If $\xi=\{\xi^{\alpha}\}$ and $\eta=\{\eta^{\alpha}\}$ are tangent vectors to $M$ at a point $z$,
then the value of this 2-form on tangent vectors is
$$\frac{\partial^2 F(z)}{\partial z^{\alpha}\partial \bar{z}^{\beta}}\xi^{\alpha}\bar{\eta}^{\beta}.$$
Since on a complex manifold the transition maps are holomorphic,
the sign of $\frac{\partial^2 F(z)}{\partial z^{\alpha}\partial \bar{z}^{\beta}}\xi^{\alpha}\bar{\xi}^{\beta}$ is independent of the choice of the holomorphic coordinates.
As a consequence, if $\frac{\partial^2 F(z)}{\partial z^{\alpha}\partial \bar{z}^{\beta}}\xi^{\alpha}\bar{\xi}^{\beta}\geq 0$ for any $\xi=\{\xi^{\alpha}\}$, we say that $F$ is  pluri-subharmonic at $z$, and if $\frac{\partial^2 F(z)}{\partial z^{\alpha}\partial \bar{z}^{\beta}}\xi^{\alpha}\bar{\xi}^{\beta} > 0$ for any $\xi=\{\xi^{\alpha}\}\neq 0$, we say that $F$ is strictly pluri-subharmonic.

Pluri-subharmonic function is a natural generalization of subharmonic function of single complex variable,
and it is related to several important notions in several complex variable theory, such as domain of holomorphy and so on.

\begin{remark}
Pulling back a convex functional by a harmonic map, one get a subharmonic function, while
pulling back a pluri-subharmonic functional by a holomorphic map, one get a subharmonic function.
Note that a holomorphic map into a K\"ahler manifold is harmonic.\end{remark}

 Let $\gamma$ be a simple closed curve on $S$ and let $E(\cdot): \mathcal{T}(S) \to \mathbb{R}$
be the energy of equivariant harmonic maps from Riemann surfaces  to the $\mathbb{R}$-tree associated with $\gamma$.
We restate Theorem \ref{theorem:WP} in the following way:

\begin{theorem}\label{theorem:sum}
Let $(S,g)\in \mathcal{T}$ such that $E(\cdot)$ is smooth in a neighborhood of $(S,g)$.
For any two Weil-Petersson geodesics $\mu_k(t),k=1,2$ on $\mathcal{T}(S)$ with $\mu_1(0)=\mu_2(0)=(S,g)$ and $\frac
{d}{dt}|_{t=0}\mu_1(0)={\mu}$, $\frac {d}{dt}|_{t=0}\mu_2(0)=i{\mu}$,
we have
$$\frac{d^2} {dt^2}|_{t=0}E(\mu_1(t))+\frac{d^2}{dt^2}|_{t=0}E(\mu_2(t))>0.$$
\end{theorem}
Note that here we have identified the family of Beltrami differentials $\mu_k (t), k=1,2$
as two Weil-Petersson paths in $\mathcal{T}(S)$,
which are corresponding to two  families of hyperbolic metrics $(S,g_t |dz_t|^2)$
 satisfying $\bar{\partial} z_t =\mu_k(t)\partial z_t, k=1,2 $, respectively.
 The proof of  Theorem \ref{theorem:sum} is given in \S \ref{sec:proof} and \S \ref{sec:strict}.

\subsection{An outline of the argument} The idea is that, by  Theorem \ref{theorem:second},
we can write down the sum $$\frac{d^2} {dt^2}|_{t=0}E(\mu_1(t))+\frac{d^2}{dt^2}|_{t=0}E(\mu_2(t))$$ into
(see \S \ref{sec:proof} below for notation)

$$\{ D^2_{22}E[{\mu}, {\mu}]+D^2_{22}E[i{\mu}, i{\mu}]\}- \{ D^2_{11}E[\dot{\omega}(\mu),
\dot{\omega}(\mu)]+D^2_{11}E[i{\mu}, i{\mu}]\}.$$
Then we apply Equation \eqref{equ11} and the Cauchy-Schwarz inequality to check that $$D^2_{11}E[\dot{\omega}(\mu),
\dot{\omega}(\mu)]+D^2_{22}E[i{\mu}, i{\mu}]\leq D^2_{22}E[{\mu}, {\mu}]+D^2_{22}E[i{\mu}, i{\mu}].$$ The key is that,
although we could not show that
 $D^2_{11}E[\dot{\omega}(\mu),
\dot{\omega}(\mu)]\leq D^2_{22}E[{\mu}, {\mu}]$ or $D^2_{11}E[i{\mu}, i{\mu}]\leq D^2_{22}E[i{\mu}, i{\mu}]$,
we can show that the summation of $D^2_{11}E[\dot{\omega}(\mu),
\dot{\omega}(\mu)]$ and $D^2_{22}E[i{\mu}, i{\mu}]$ will cancel some non-trivial terms.

To see that the above inequality is strict,
we consider the situation when the Cauchy-Schwarz inequality is an equality: by using the fact that
 any holomorphic vector field on a compact Riemann surface is zero, we show that
 in this case the variational vector $\dot\omega$ should be vanished.

\subsection{Proof of  Theorem \ref{theorem:sum}}\label{sec:proof}

Let $\mu_k(t), k=1,2$ be  two Weil-Petersson geodesics satisfying $\mu_1(0)=\mu_2(0)$ and $\frac
{d}{dt}\mu_1(0)={\mu}$, $\frac {d}{dt}\mu_2(0)=i{\mu}$. Corresponding to the two families of hyperbolic metrics,
there are two families of equivariant harmonic maps
whose energy realize the extremal length functions.
We  denote by $\dot{\omega}(\mu)$ and $\dot{\omega}(i\mu)$
 the variational vectors of harmonic maps corresponding to the direction $\mu$ and $i\mu$, respectively.

By our discussion in \S 3, we have
\begin{equation*}\label{equ13}
\begin{cases}
\frac{d^2} {dt^2} |_{t=0}E(\mu_1(t))
=-D^2_{11}E[\dot{\omega}(\mu),
\dot{\omega}(\mu)]+D^2_{22}E[{\mu}, {\mu}],\\
\\
\frac{d^2} {dt^2} |_{t=0} E(\mu_2(t))
=-D^2_{11}E[\dot{\omega}(i\mu),
\dot{\omega}(i\mu)]+D^2_{22}E[i{\mu}, i{\mu}].
\end{cases}
\end{equation*}

Applying Theorem \ref{theorem:second}, we have
\begin{equation*}
\begin{cases}
D^2_{22}E[{\mu}, {\mu}]
=2 \int_S |\mu|^2(|\omega_z|^2+|\omega_{\bar z}|^2) \ dzd \bar z, \\
\\
D^2_{22}E[\dot{\omega}(\mu),\dot\omega(\mu)]=2\int_S |\frac{\partial \dot \omega(\mu)}{\partial z}|^2 \ dzd \bar z.
\end{cases}
\end{equation*}
In the second equation, $\frac{\partial \dot \omega(\mu)}{\partial z}$ denotes the partial derivative
of $\dot \omega(\mu)$ with respect to $z$ (while in previous sections, when $\dot\omega=\dot\omega(\mu)$,
 it was denoted by $\dot\omega_z$ for simplicity). The expressions for $D^2_{11}E[\dot{\omega}(i\mu),
\dot{\omega}(i\mu)]$ and $D^2_{22}E[i{\mu}, i{\mu}]$ are similar.

As a result, we have
\begin{eqnarray}\label{equ14}
&& \frac{d^2} {dt^2}|_{t=0}E(\mu_1(t))+\frac{d^2} {dt^2}|_{t=0}E(\mu_2(t)) \nonumber\\
&=& 4 \int_S |\mu|^2(|\omega_z|^2+|\omega_{\bar z}|^2)dzd \bar z \\
&& - 2\int_S |\frac{\partial \dot \omega(\mu)}{\partial z}|^2 dzd \bar z-2\int_S |\frac{\partial \dot \omega(i\mu)}{\partial z}|^2 dzd \bar z. \nonumber
\end{eqnarray}

By equation $(\ref{equ11})$, we have
\begin{eqnarray*}
\int_S |\frac{\partial \dot \omega(\mu)}{\partial z}|^2 dzd\bar
z=\int_S \frac{\partial \overline{\dot \omega(\mu)}}{\partial z} \mu\omega_z+
\int_S \frac{\partial \overline{\dot \omega(\mu)}}{\partial \bar z} \overline{\mu}\omega_{\bar z}dzd \bar
z ,
\end{eqnarray*}
and
\begin{eqnarray*}
\int_S |\frac{\partial \dot \omega(i\mu)}{\partial z}|^2 dzd\bar
z&=&\int_S \frac{\partial \overline{\dot \omega(i\mu)}}{\partial z}  (i\mu) \omega_z+
\int_S \frac{\partial \overline{\dot \omega(i\mu)}}{\partial \bar z} \overline{(i\mu)}\omega_{\bar z}dzd \bar z \\
&=&\int_S i \frac{\partial \overline{\dot \omega(i\mu)}}{\partial z}  \mu \omega_z+
\int_S  -i\frac{\partial \overline{\dot \omega(i\mu)}}{\partial \bar z}\overline{\mu}\omega_{\bar z}dzd \bar z.
\end{eqnarray*}

Since $\omega$ is real, we have
\begin{eqnarray*}
&& 2\int_S  |\frac{\partial \dot \omega(\mu)}{\partial z}|^2 dzd\bar z+2\int |\frac{\partial \dot \omega(i\mu)}{\partial z}|^2 dzd \bar z \\
&=& 2\int_S \mu\omega_z[\frac{\partial \overline{\dot \omega(\mu)}}{\partial z}
+i\frac{\partial \overline{\dot \omega(i\mu)}}{\partial z}]dzd \bar
z+2\int_S \bar{\mu}\omega_{\bar
z}[\frac{\partial \overline{\dot \omega(\mu)}}{\partial \bar z}-i\frac{\partial \overline{\dot \omega(i\mu)}}{\partial \bar z}]dzd
\bar z \\
&\leq&  2\int_S  |\mu|^2(|\omega_z|^2+|\omega_{\bar z}|^2)dzd \bar z\\
&& + \frac{1}{2}\int_S |\frac{\partial \overline{\dot \omega(\mu)}}{\partial  z}+i\frac{\partial \overline{\dot \omega(\mu)}}{\partial  z}|^2dzd \bar
z+\frac{1}{2}\int_S |\frac{\partial \overline{\dot \omega(\mu)}}{\partial \bar z}-i\frac{\partial \overline{\dot \omega(i\mu)}}{\partial \bar z}|^2dzd \bar z \\
&=& 2\int_S  |\mu|^2(|\omega_z|^2+|\omega_{\bar z}|^2)dzd \bar z
+\int_S  |\frac{\partial {\dot \omega(\mu)}}{\partial z}|^2 dz d \bar z+\int_S  |\frac{\partial {\dot \omega(i\mu)}}{\partial  z}|^2 dzd \bar z \\
 && -i \  \mathrm{Re} \big( \int_S \frac{\partial \overline{\dot \omega(\mu)}}{\partial  z}
 \frac{\partial {\dot \omega(i\mu)}}{\partial \bar z} dzd \bar z
- \int_S \frac{\partial \overline{\dot \omega(\mu)}}{\partial  \bar z}
 \frac{\partial {\dot \omega(i\mu)}}{\partial z}dzd \bar z\big). \label{equ}
\end{eqnarray*}
In the above computation, we have applied the following inequality:
$$2\int_S |fg|dz d\bar z\leq 2\int |f|^2 dz d\bar z+ \frac{1}{2}\int |g|^2dz d\bar z.$$

It follows from integration by
part that

 \begin{eqnarray*}
&& \int_S \frac{\partial \overline{\dot \omega(\mu)}}{\partial  z}
 \frac{\partial {\dot \omega(i\mu)}}{\partial \bar z} dzd \bar z
-  \int_S \frac{\partial \overline{\dot \omega(\mu)}}{\partial  \bar z}
 \frac{\partial {\dot \omega(i\mu)}}{\partial z}dzd \bar z  \\
 &=&-\int_S \frac{\partial^2 \overline{\dot \omega(\mu)}}{\partial  z\partial \bar z}
{{\dot \omega(i\mu)}} dzd \bar z
+ \int_S \frac{\partial^2 \overline{\dot \omega(\mu)}}{\partial z \partial  \bar z}
{\dot \omega(i\mu)}dzd \bar z  \\
&=& 0.
 \end{eqnarray*}

 Therefore we have
\begin{eqnarray*}
&& 2\int_S  |\frac{\partial \dot \omega(\mu)}{\partial z}|^2 dzd\bar z+2\int |\frac{\partial \dot \omega(i\mu)}{\partial z}|^2 dzd \bar z \\
 &\leq& 2\int_S  |\mu|^2(|\omega_z|^2+|\omega_{\bar z}|^2)dzd \bar z
+\int_S  |\frac{\partial {\dot \omega(\mu)}}{\partial z}|^2 dz d \bar z+\int_S  |\frac{\partial {\dot \omega(i\mu)}}{\partial  z}|^2 dzd \bar z.
\end{eqnarray*}
  Then
\begin{eqnarray*}
&& \int_S  |\frac{\partial \dot \omega(\mu)}{\partial z}|^2 dzd\bar z+2\int |\frac{\partial \dot \omega(i\mu)}{\partial z}|^2 dzd \bar z \\
&\leq& 2\int |\mu|^2(|\omega_z|^2+|\omega_{\bar z}|^2)dzd \bar z.
\end{eqnarray*}
Combined with $(\ref{equ14})$, we conclude that
\begin{equation}\label{equ:plurisub}
\frac{d^2} {dt^2}|_{t=0}E(\mu_1(t))+\frac{d^2} {dt^2}|_{t=0}E(\mu_2(t))\geq 0.
\end{equation}

\subsection{Strict inequality}\label{sec:strict}

Next we show that the left hand side of $(\ref{equ:plurisub})$ is actually positive.

As we have mentioned above, the Jacobian
$|\omega_z|^2-|\omega_{\bar z}|^2=0.$  Moreover we can take
$\omega$ to be real and then $\omega_{\bar
z}=\overline{\omega_z},\dot\omega_{\bar z}=\overline
{\dot\omega_z}.$

Note that in the proof of pluri-subharmonicity, we have applied the Schwarz inequality to show that
\begin{eqnarray*}
&& 2\int_S  |\frac{\partial \dot \omega(\mu)}{\partial z}|^2 dzd\bar z+2\int |\frac{\partial \dot \omega(i\mu)}{\partial z}|^2 dzd \bar z \\
&=& 2\int_S \mu\omega_z[\frac{\partial \overline{\dot \omega(\mu)}}{\partial z}
+i\frac{\partial \overline{\dot \omega(i\mu)}}{\partial z}]dzd \bar
z+2\int_S \bar{\mu}\omega_{\bar
z}[\frac{\partial \overline{\dot \omega(\mu)}}{\partial \bar z}-i\frac{\partial \overline{\dot \omega(i\mu)}}{\partial \bar z}]dzd
\bar z \\
&&  \leq2\int_S  |\mu|^2(|\omega_z|^2+|\omega_{\bar z}|^2)dzd \bar z\\
&& +\frac{1}{2}\int_S |\frac{\partial \overline{\dot \omega(\mu)}}{\partial  z}+i\frac{\partial \overline{\dot \omega(i\mu)}}{\partial  z}|^2dzd \bar
z+\frac{1}{2}\int_S |\frac{\partial \overline{\dot \omega(\mu)}}{\partial \bar z}-i\frac{\partial \overline{\dot \omega(i\mu)}}{\partial \bar z}|^2dzd \bar z \\
&=& 2\int_S  |\mu|^2(|\omega_z|^2+|\omega_{\bar z}|^2)dzd \bar z
+\int_S  |\frac{\partial {\dot \omega(\mu)}}{\partial z}|^2 dz d \bar z+\int_S  |\frac{\partial {\dot \omega(i\mu)}}{\partial \bar z}|^2 dzd \bar z \\
\end{eqnarray*}

We denote the above inequality by $(\ast)$. Obviously,  if the inequality $(*)$ is a strictly inequality,
 then by formula (11)  the extremal length is strictly pluri-subharmonic.
Thus we assume that the inequality $(*)$ is an equality.

If $\omega_z=0$ at a point, then $\omega_{\bar z}=0 $  at this point too. Let $S_0=\{z \in S: \omega_z=0\}$. Then the integration
 $$\int_{S_0} |\mu|^2(|\omega_z|^2+|\omega_{\bar
z}|^2)dzd \bar z=0,$$
$$\int_{S_0} |\frac{\partial \dot \omega(\mu)}{\partial z}|^2 dzd\bar
z=\int_{S_0} \frac{\partial \overline{\dot \omega(\mu)}}{\partial z} \mu\omega_z+
\int_{S_0} \frac{\partial \overline{\dot \omega(\mu)}}{\partial \bar z}\bar \mu \omega_{\bar z}dzd \bar
z=0,$$
 and
 $$\int_{S_0} |\frac{\partial \dot \omega(i\mu)}{\partial z}|^2 dzd\bar z=0.$$

 As a result,
 the points in $S_0$  have no  contribution to the  inequality $(*)$
and the inequality
\begin{eqnarray*}
&& \int_S  |\frac{\partial \dot \omega(\mu)}{\partial z}|^2 dzd\bar z+2\int |\frac{\partial \dot \omega(i\mu)}{\partial z}|^2 dzd \bar z \\
&\leq& 2\int_S |\mu|^2(|\omega_z|^2+|\omega_{\bar z}|^2)dzd \bar z.
\end{eqnarray*}

Therefore, we only need to consider the   inequality $(*)$ in
any neighborhood $\mathcal{U}\subset S$ where $\omega_z\neq0$.

Suppose that $\omega_z\neq0 $ in $\mathcal{U}$, Let $$\tau(z)=\frac
{\dot\omega(\mu)-i\dot\omega(i\mu)}{\omega_z}.$$
Since
$\omega$ is real, we have
$$\overline{\tau}(z)=\frac
{\dot\omega(\mu)+i\dot\omega(i\mu)}{\omega_{\bar z}}.$$
Since $\omega_{z\bar z}=0$, we have
$$\tau_{\bar z} = \left( \frac{\partial \dot \omega(\mu)}{\partial \bar z}-i \frac{\partial \dot \omega(i\mu)}{\partial \bar z} \right) / {\omega_z} $$
and $$
\bar\tau_z=\left( \frac{\partial \dot \omega(\mu)}{\partial  z}+ i \frac{\partial \dot \omega(i\mu)}{\partial z} \right)/ {\omega_{\bar z}}.$$
Note that the inequality $(*)$ becomes equality only if
$$\tau_{\bar z}\omega_z=f_1(z)\bar\mu\omega_z , \bar\tau_z\omega_{\bar z}=f_2(z)
\bar\mu\omega_{\bar z},
$$
where $f_1(z)$ and $f_2(z)$ are real-valued
functions defined on $\mathcal{U}$.

Assume that the inequality $(*)$ becomes equality. Since $\omega_z\neq0$ and $\omega_{\bar z}\neq0$ in $\mathcal{U}$, we have
$\tau_{\bar z}=f_1(z)\bar\mu , \bar\tau_z=f_2(z) \bar\mu.$
In this case we claim that $\tau_{\bar z}=0$ in $\mathcal{U}$.

In fact,
$$\tau_{\bar z}=f_1(z)\bar\mu=f_2(z)\mu.$$
By assumption, $\mu=\frac{\bar \Phi}{g}$ where $\Phi$ is a holomorphic quadratic differential.
If $z$ is a point in $\mathcal{U}$ such that $f_1(z)\neq 0$ (and then $f_2(z)\neq 0$), then we have
$$\frac{\Phi^2(z)}{|\Phi(z)|^2}=\frac{\Phi(z)}{\bar \Phi(z)}=\frac{f_2(z)}{f_1(z)}.$$
We claim that $z$ should be a zero point of $\Phi$. Otherwise, the value of $\Phi^2(z)$ is real in a neighborhood of $z$, which is impossible
(note that $\Phi^2(z)$ is holomorphic). As a result, $\tau_{\bar z}=0$ in $\mathcal{U}$.

Since the Hopf differential of the harmonic map $\omega$
is non-trivial, $\omega_z\neq 0$ on the Riemann surface except  a finite set. We can cover the Riemann surface $(S,g)$
by a family of open sets $\mathcal{U}$ satisfying $\omega_z\neq 0$ for $z\in \mathcal{U}$. It turns out that
 $\tau_{\bar z}=0$ on the Riemann surface, that is, $\tau$ is a holomorphic vector field on $(S,g)$ (since
 $\omega_z dz$ is a differential, $\tau dz^{-1}$ is a negative differential on $(S,g)$).
It is well known that the only holomorphic vector field is zero, thus $\tau\equiv 0$.
This implies that $\dot\omega(\mu)=\dot\omega(i\mu)\equiv 0$.
In this case we have (recall the first equality in $(*)$)
$$0=\int_{\mathcal{U}}  |\frac{\partial \dot \omega(\mu)}{\partial z}|^2 dzd\bar z+\int_{\mathcal{U}} |\frac{\partial \dot \omega(i\mu)}{\partial z}|^2 dzd \bar z < 2\int_{\mathcal{U}}
|\mu|^2(|\omega_z|^2+|\omega_{\bar z}|^2)dzd \bar z.$$

By formula \eqref{equ14}, we have proved the strictly inequality in Theorem \ref{theorem:sum}.

\subsection{Pluri-subharmonicity of extremal length}
In general, an upper semi-continuous function $F$ on a complex manifold $M$ is said to be {pluri-subharmonic}
 if  for any holomorphic map $\phi: \Delta \to M$, the function
$F\circ \phi: \Delta \to M$
is subharmonic, where  $\Delta$ denotes the unit disk. This is equivalent to
the property

$$F\circ \phi (a)\leq \frac{1}{2\pi} \int_{0}^{2\pi} F\circ \phi(a+r e^{i\theta}) \ d \theta, \forall \ a\in \Delta, r< 1-|a|.$$
In the sense of  de Rham, $d d^c F$ defines a positive $(1, 1)$-current on $M$.

Theorem \ref{theorem:main} follows from Theorem \ref{theorem:WP} and Lemma \ref{lem:denseopen}, combined with the
following theorem.

\begin{theorem}\label{thm:foliation}
Given any measured foliation $\mathcal{F}$ on $S$, the extremal length function $\mathrm{Ext}_{\mathcal{F}}(\cdot)$ is pluri-subharmonic on $\mathcal{T}(S)$.
\end{theorem}
\begin{proof}
Let $X\in \mathcal{T}(S)$.
We  first assume that $\mathcal{F}$ is equivalent to a complete measured geodesic lamination.
(A geodesic lamination is \emph{complete} if its complement are disjoint union of ideal triangles).
In this case the Hubbard-Masur differential of $\mathcal{F}$ always has only simple zeros.
Then we take a sequence of weighed simple closed curves $\{a_n\gamma_n\}$ to approximate $\mathcal{F}$.
When $n$ is sufficiently large, the quadratic differential $\Phi_n$ on $X$ with vertical foliations equivalent
to $\{a_n\gamma_n\}$ also has only simple zeros in a neighborhood of $X$. As a result, the extremal length of
$a_n\gamma_n$ is smooth near $X$. By Theorem \ref{theorem:sum}, the extremal length of
$a_n\gamma_n$ is strictly pluri-subharmonic at $X$. By taking a limit, we know that  $\mathrm{Ext}_{\mathcal{F}}(\cdot)$ defines a pluri-subharmonic
function in a neighborhood of $X$. Since $X$ is arbitrary, we conclude that  $\mathrm{Ext}_{\mathcal{F}}(\cdot)$ is pluri-subharmonic.

In general, note that the subset of essentially complete measured geodesic lamination is dense in $\mathcal{MF}$. By approximation,
it follows that the extremal length function of any measured foliation  is pluri-subharmonic on $\mathcal{T}(S)$, in the sense of current.

\end{proof}

\begin{remark}
Using a different method, Vasil'ev \cite{Vasi} showed
that the extremal length  of some maximally rational measured foliations
are locally harmonic in $\mathcal{T}(S)$.

As we have observed in \S \ref{sec:regularity}, for any simple closed curve $\gamma$ on $S$,
there is an open dense subset $\mathcal{T}_\gamma$ of $\mathcal{T}(S)$ on which the extremal length $\mathrm{Ext}_{\gamma}(\cdot)$ is
strictly pluri-subharmonic. It would be interesting to understand the complement of  $\mathcal{T}_\gamma$.
\end{remark}

\section{Some Applications}\label{sec:application}
We  obtain a new proof of the following result of Bers-Ehrenpreis \cite{Bers} (see also \cite{Tromba,Wolpert87}).
\begin{corollary}
Teichm\"uller space is a Stein manifold.
\end{corollary}
\begin{proof}
Let $\{\gamma_i\}$ be a finite set of simple closed curves filling the surface $S$ and set
$$L(\cdot)=\sum_{\gamma_i}\mathrm{Ext}_{\gamma_i}(\cdot).$$
By definition of extremal length, $L(\cdot)$ has a universal lower bound
\begin{equation}\label{equ:lower}
\sum_{\gamma_i}{\ell^2_{\gamma_i}(\cdot)}/{2\pi|\chi(S)|},
\end{equation}
where $\ell_{\gamma_i}(\cdot)$ denotes the hyperbolic length of $\gamma_i$.
Since $(\ref{equ:lower})$ is a proper function on $\mathcal{T}(S)$ (see Kerckhoff \cite{Kerckhoff83}), $L(\cdot)$ is also proper.
It follows from Theorem \ref{theorem:main} that $L(\cdot)$ defines a proper plurisubharmonic function on $\mathcal{T}(S)$. This implies that $\mathcal{T}(S)$ is a Stein manifold.

\end{proof}

\subsection{Relation with the Teichm\"uller metric}
In the following, we apply $(\ref{equ11})$ to present an inequality about the second variation of $E(\omega^t,g_t)=\mathrm{Ext}_\gamma(g_t)$ along a Teichm\"uller geodesic $(S,g_t)$ (note that for a Teichm\"uller geodesic, we can assume that the Beltrami differentials
of the quasiconformal mappings $(S,g_0)\to (S,g_t)$ satisfy $\mu(t)=t\mu$, where $\|\mu\|_{\infty}=1$).
As before, we assume that $\mathrm{Ext}_\gamma(g_t)$ is smooth.

\begin{proposition}\label{pro:second}
Let $(S,g_t)$ be a Teichm\"uller geodesic in
$\mathcal{T}(S)$ with Beltrami differential $\mu(t)=t\mu$, where $\|\mu\|_{\infty}=1.$ Then the following inequality holds:
\begin{eqnarray*}
\frac{d^2} {dt^2}|_{t=0}E(\omega^t,g_t)
&\geq&  -2 \int_S (|\omega_z|^2 +|\omega_{\bar z}|^2)dzd \bar z \\
&=& -4 \int_S |\omega_z|^2 dzd \bar z=-4E(\omega,g_0).
\end{eqnarray*}
\end{proposition}
\begin{proof}
It is not hard to see that the discussions in \S \ref{sec:second} and \S \ref{sec:field} apply to
second variations along Teichm\"uller
geodesics.
As before, we can assume that $\omega$ is real, and it follows that $\dot\omega_z=\overline{\dot\omega}_z$. By $(\ref{equ11})$, we have
\begin{eqnarray*}
\int|\dot \omega_z|^2 dzd \bar z
&=&|\int_S \overline{\dot\omega}_z \mu\omega_z+
\int_S \overline{\dot\omega_z}\overline{\mu}\omega_{\bar z}dzd \bar
z| \\
&\leq& \int_S |\overline{\dot\omega}_z \mu\omega_z|dzd \bar z+
\int_S |\overline{\dot\omega_z}\overline{\mu}\omega_{\bar z}|dz d
\bar z \\
&=&\int_S   |\dot\omega_z||\mu|(|\omega_z|+|\omega_{\bar z}|)dzd \bar z \\
&\leq& \frac 1 2 \int_S |\dot\omega_z|^2dzd \bar z
+\frac 1 2 \int_S |\mu|^2(|\omega_z|+|\omega_{\bar z}|)^2 dzd \bar
z.
\end{eqnarray*}
Then
$$\frac 1 2 \int|\dot\omega_z|^2dzd \bar z
\leq \frac 1 2 \int|\mu|^2(|\omega_z|+|\omega_{\bar z}|)^2 dzd
\bar z.$$ Equivalently,
$$\int_S |\dot\omega_z|^2dzd \bar z
\leq \int_S |\mu|^2(|\omega_z|+|\omega_{\bar z}|)^2 dzd \bar z.$$ As
the image of $\omega$ is $\mathbb{R}$-tree, the Jacobian is zero. We have
$$0=|\omega_z|^2 -|\omega_{\bar z}|^2.$$
Then

\begin{eqnarray*}
\int_S |\dot\omega_z|^2dzd \bar z &\leq& \int_S |\mu|^2(2|\omega_z|)^2 dzd \bar z \\
&\leq& 4 \int_S |\mu|^2|\omega_z|^2 dzd \bar z \\
&=& 2 \int_S |\mu|^2(|\omega_z|^2+|\omega_{\bar z} |^2)dzd \bar z.
\end{eqnarray*}
Combining the above inequality with Theorem \ref{theorem:second}, we have

\begin{eqnarray*}
\frac{d^2} {dt^2}|_{t=0}E(\omega^t,g_t)&=&\frac{d^2} {dt^2}|_{t=0}E(\omega,g_t)
-\frac{d^2} {dt^2}|_{t=0}E(\omega^t,g_0) \\
&=& 2 \int_S |\mu|^2(|\omega_z|^2+|\omega_{\bar z}|^2)dzd \bar z
-2\int_S |\dot \omega_z|^2 dzd \bar z \\
&\geq& 2 \int_S |\mu|^2 (|\omega_z|^2 +|\omega_{\bar z}|^2)dzd \bar z \\
&& -4\int_S |\mu|^2 (|\omega_z|^2 +|\omega_{\bar z}|^2)dzd \bar z \\
&=& -2 \int_S |\mu|^2 (|\omega_z|^2 +|\omega_{\bar z}|^2)dzd \bar z.
\end{eqnarray*}
By assumption, $\|\mu\|_\infty=1$, thus
$$\frac{d^2} {dt^2}|_{t=0}E(\omega^t,g_t)\geq -2 \int_S  (|\omega_z|^2 +|\omega_{\bar z}|^2)dzd \bar z=-4E(\omega,g_0).$$
\end{proof}

 It was a long open problem that whether Teichm\"uller geodesic balls are convex. 
  Kerckhoff \cite{Kerckhoff} has discovered an elegant and useful way to
compute the Teichm\"{u}ller distance in terms of extremal length.
Lenzhen and Rafi \cite{LR}
proved that  extremal length functions are quasi-convex along any Teichm\"uller geodesic.
As a corollary, they proved the quasi-convexity of Teichm\"uller geodesic balls.
 Recently, Bourque and Rafi \cite{BR} announced non-convexity of Teichm\"uller metric balls.
 It is interesting to contrast the above result of Lenzhen-Rafi to a corollary of Chaika-Masur-Wolf \cite{CMW} that
the hyperbolic length can increase and decrease along Teichm\"uller geodesics.
We hope that our study of extremal length variations may be applied in the study of the convexity of Teichm\"uller metric.

\end{document}